\documentclass[11pt,a4paper]{article}
\usepackage[margin=1in]{geometry}
\usepackage{tabularx, enumerate, authblk,soul}
\usepackage[dvipsnames]{xcolor}
\usepackage{amsmath,amssymb,dsfont,subcaption}
\usepackage{mathtools, graphicx,color,dsfont,amsthm,bbm}
\usepackage{lipsum,bm,comment}
\usepackage{natbib}
\usepackage{longtable}
\usepackage{array}
\usepackage{caption}
\usepackage[colorlinks=true,linkcolor=blue,anchorcolor=blue,citecolor=blue,filecolor=black,menucolor=black,runcolor=black,urlcolor=black]{hyperref}
\usepackage{cleveref}

\usepackage{algorithm}
\usepackage{algpseudocode}
\algnewcommand\algorithmicinput{\textbf{INPUT:}}
\algnewcommand\INPUT{\item[\algorithmicinput]}
\algnewcommand\algorithmicoutput{\textbf{OUTPUT:}}
\algnewcommand\OUTPUT{\item[\algorithmicoutput]}

\newcommand{\RN}[1]{%
  (\textup{\uppercase\expandafter{\romannumeral#1}})%
}
\newcolumntype{H}{>{\setbox0=\hbox\bgroup}c<{\egroup}@{}}

\theoremstyle{plain}
\newtheorem{thm}{Theorem}
\newtheorem{lemma}[thm]{Lemma}

\newtheorem{prop}[thm]{Proposition}

\theoremstyle{definition}
\newtheorem{defn}{Definition}
\newtheorem{exmp}{Example}
\newtheorem{remark}{Remark}

\DeclareMathOperator*{\argmax}{arg\,max}

\allowdisplaybreaks
\newcommand{\Mod}[1]{\ (\mathrm{mod}\ #1)}

\title{On robustness and local differential privacy}
\author{Mengchu Li}
\author{Thomas B.~Berrett}
\author{Yi Yu}
\affil{Department of Statistics, University of Warwick}
\date{\today}

\begin{document}

\maketitle

\begin{abstract}
It is of soaring demand to develop statistical analysis tools that are robust against contamination as well as preserving individual data owners' privacy. In spite of the fact that both topics host a rich body of literature, to the best of our knowledge, we are the first to systematically study the connections between the optimality under Huber's contamination model and the local differential privacy (LDP) constraints.  

In this paper, we start with a general minimax lower bound result, which disentangles the costs of being robust against Huber's contamination and preserving LDP.  We further study four concrete examples: a two-point testing problem, a potentially-diverging mean estimation problem, a nonparametric density estimation problem and a univariate median estimation problem.  For each problem, we demonstrate procedures that are optimal in the presence of both contamination and LDP constraints, comment on the connections with the state-of-the-art methods that are only studied under either contamination or privacy constraints, and unveil the connections between robustness and LDP via partially answering whether LDP procedures are robust and whether robust procedures can be efficiently privatised. Overall, our work showcases a promising prospect of joint study for robustness and local differential privacy. 

\end{abstract}

\section{Introduction}\label{sec-introduction}

In modern data collection and analysis, the privacy of individuals is a key concern. There has been a surge of interest in developing data analysis methodologies that yield strong statistical performance without compromising individuals' privacy, largely driven by applications in modern technology, including in Google \citep[e.g.][]{erlingsson2014rappor}, Apple \citep[e.g.][]{tang2017privacy} and Microsoft \citep[e.g.][]{ding2017collecting}, and by pressure from regulatory bodies \citep[e.g.][]{forti2021deployment,aridor2021effect}. The prevailing framework for the development of private methodology is that of differential privacy \citep{dwork2006calibrating}.  Although this originates in cryptography, there is a growing body of statistical literature that aims to explore the constraints of this framework and provide procedures that make optimal use of available data \citep[e.g.][]{wasserman2010statistical, duchi2018minimax, rohde2020geometrizing, cai2021cost}. Work in this area is split between central models of privacy, where there is a third party trusted to collect and analyse data before releasing privatised results, and local models of privacy, where data are randomised before collection.  We, in this paper, will consider the local differential privacy constraint, to be formally defined in \Cref{sec-general-setup}. While classical methods for locally private analysis are restricted to the estimation of the parameter of a binomial distribution \citep{warner1965randomized}, modern research has resulted in mechanisms for many other statistical problems including various hypothesis testing problems \citep[e.g.][]{kairouz2014extremal, joseph2019role, berrett2020locally, acharya2021interactive, lam2020minimax}, mean and median estimation \citep[e.g.][]{duchi2018minimax}, nonparametric estimation problems \citep[e.g.][]{rohde2020geometrizing, butucea2020local}, and change point analysis \citep[e.g.][]{berrett2021changepoint,li2022network}, to name but a few.

In addition to preserving individuals' privacy, being robust to outliers and adversarial contamination is another desideratum for modern learning algorithms. The study of robust statistical procedures has sparked great interest among statisticians and there have been a number illuminating textbooks in this area \citep[e.g.][]{huber1981robust,hampel2011robust,huber2004robust,maronna2019robust}. The focus of classical robust statistics is on the analysis of outliers' influence on statistical procedures, quantified through specific notions such as the breakdown point and influence function. Due to the demands of analysing more complex data types, the focus has, more recently, shifted towards providing non-asymptotic guarantees on convergence rates under various contamination models, including heavy-tailed models \citep[e.g.][]{catoni2012challenging,lugosi2019sub}, different types of model mis-specification \citep[e.g.][]{huber1968robust, cherukuri2020consistency}, and strong contamination models \citep[e.g.][]{diakonikolas2017being, lugosi2021robust,pensia2020robust}.  A number of computationally-efficient algorithms that achieve (nearly) minimax rate-optimal rates under either one or more aforementioned models have also been proposed for various tasks; see \cite{diakonikolas2019robust} for a recent survey. 

The connections between differential privacy and robustness have been well studied in the central model of differential privacy, where there is a trusted data curator. A natural starting point for many differentially private estimators is a function of the data whose sensitivity to changes in single observations can be controlled \citep[e.g.][]{dwork2006calibrating, dwork2009differential, canonne2019structure, cai2021cost}. This is also the case for robust statistics, where estimators are often constructed in order to be minimally sensitive to arbitrary changes in a small number of data points \citep[e.g.][]{huber1981robust,huber2004robust}. See \cite{avella2020role} for further discussion of these connections. There has been, recently, an increasing trend, mostly in the theoretical computer science literature, of developing algorithms that are simultaneously robust and privacy preserving  under the central model \citep[e.g.][]{dimitrakakis2014robust,ghazi2021robust,esfandiari2021tight,kothari2021private,liu2021robust}. There is, however, little work on the connection between privacy and robustness in the local model of differential privacy. A key feature to set apart our work from the aforementioned robust procedures in the central models is that in the central models it is possible to add noise after computing robust estimators, but in local privacy the requirement to add noise to each observation separately means that the approaches taken in the two models are fundamentally different. Note that some very recent works \citep{cheu2021manipulation,acharya2021robust,chhor2022robust} consider contamination after the privatisation step and the results therein feature an interaction of privacy level $\alpha$ and contamination level $\varepsilon$. In our work, we suppose that contamination happens before the data are sent for privatisation.

\subsection{A summary of our contributions}\label{sec-contributions}

This paper concerns the pursuit of answers to the questions:
\begin{itemize}
    \item [Q1.] \emph{Can robust procedures be directly applied to privatised information and attain optimal performance?}
    \item [Q2.] \emph{Can locally private procedures be automatically robust?}
\end{itemize}
To address the aforementioned questions, in this work we will study a range of statistical problems, including hypothesis testing (\Cref{sec1}), mean estimation (\Cref{sec2}), nonparametric density estimation (\Cref{sec3}) and median estimation (\Cref{secE}) with data assumed to be generated according to Huber's $\varepsilon$-contamination model \citep[e.g.][]{huber1992robust}, specified in \eqref{hubermodel}.

As for Q1, when studying Huber's contamination model without privacy constraints, \cite{chen2016general} developed a general theory and showed that a Scheff\'e tournament method provides optimal estimators for many problems. Such a method discretises the parameter space and reduces estimation problems to hypothesis selection problems. However, it was shown by \cite{gopi2020locally} that hypothesis selection is exponentially more difficult under local privacy constraints, prohibiting the use of this general learning scheme in many specific statistical learning tasks; we discuss this aspect in \Cref{sec-disc-test}.  Despite this negative answer to Q1 for this general robust procedure, for the specific problems considered in the sequel, we will show that optimal LDP procedures can be regarded as robust procedures applied to privatised data.  See Sections~\ref{sec-disc-test}, \ref{sec-disc-mean} and \ref{sec-disc-density} for details.

As for Q2, it turns out that there are deep connections between locally private estimation and estimation within Huber's contamination model. \cite{chen2016general} showed that the total variation modulus of continuity, defined in \eqref{eq-TV-modulus-conti}, controls the difficulty of a wide range of statistical problems studied under contamination. On the other hand, \cite{rohde2020geometrizing} showed that this modulus is also the key quantity in understanding the difficulty of a general class of estimation problems under local privacy. We further explore this relationship, and show that suitably-chosen procedures for locally private estimation are also optimal when Huber's contamination is introduced.  More importantly, we see that in specific problems the costs of preserving privacy and contamination are separable, which matches the intuition behind our lower bound result in \Cref{general_lower}.

In this paper, we will show that for a range of statistical problems, we are able to find procedures that are simultaneously robust, privacy-preserving and statistically rate-optimal, in terms of the contamination proportion $\varepsilon$, the privacy parameter $\alpha$, the sample size $n$ and other model parameters that may occur in specific problems.

\begin{itemize}
    \item Section~\ref{sec1} considers a simple hypothesis testing problem. When contamination is introduced, this becomes a composite hypothesis testing problem where the separation between the hypotheses depends on the level of contamination. We study a combination of the Scheff\'e test \citep{devroye2001combinatorial} and the randomised response mechanism \citep{warner1965randomized}, that results in a test that has previously been used for hypothesis testing under local differential privacy constraint \citep[e.g.][]{joseph2019role,gopi2020locally}. We obtain matching upper and lower bounds to prove that this procedure is optimal under Huber's contamination and privacy constraints.
    
    \item In Section~\ref{sec2}, we turn our attention to robust mean estimation, where the inlier distribution has bounded $k$th central moment for some fixed $k>1$, and unknown mean in $[-D,D]$ for some potentially diverging $D \geq 1$. We propose a procedure that is minimax rate-optimal in terms of the mean upper bound $D$, the sample size $n$, the privacy parameter $\alpha$ and the contamination proportion $\varepsilon$.  Previous work on mean estimation under local differential privacy \citep{duchi2018minimax} assumes that $D = 1$ and deploys a Laplace privacy mechanism, which we show is sub-optimal when $D$ is large. Previous work has noted the difficulty of private estimation with unbounded parameter spaces, both in the central model of privacy \citep{brunel2020propose,kamath2021private} and the local model \citep{duchi2013local}. In the central model, bounds on the unknown mean can be avoided with careful procedures, but in the local model consistent estimation is impossible when $D=\infty$, even without contamination \citep[see Appendix G,][]{duchi2013local}. We derive a phase transition phenomenon, whereby there exists a boundary for $D$ beyond which consistent estimation is impossible and below which a rate-optimal procedure is available.
    
    \item We study nonparametric density estimation problems in Section~\ref{sec3}. The procedures we consider are the basis expansion procedures of \cite{duchi2018minimax} and \cite{butucea2020local}. We give new analyses to show that these privacy procedures remain minimax rate-optimal with additional Huber's contamination introduced, for squared-$L_2$ and $L_\infty$ losses, respectively.
    
    \item In \Cref{secE}, we study a univariate median estimation problem, deploying a locally private stochastic gradient descent method \citep[e.g.][]{duchi2018minimax}.  Different from the privacy mechanism in the aforementioned three problems, the privacy mechanism here is sequentially interactive.  We show that this method is robust against Huber contamination and minimax rate-optimal regarding all the model parameters.
\end{itemize}

\subsection{General setup}\label{sec-general-setup}

We will now formally define the framework we study. Let $\mathcal{P}$ denote a class of distributions on the sample space $\mathcal{X}$ and let $\mathcal{G} \supset \mathcal{P}$ denote the set of all distributions on $\mathcal{X}$.  Two key ingredients in this paper are: (a) robustness against Huber's contamination and (b) privacy preservation in the sense of local differential privacy.

As for the contamination, to be specific, we consider problems where data are generated not directly from $P \in \mathcal{P}$, but from a contaminated distribution $P_\varepsilon \in \mathcal{P}_\varepsilon(\mathcal{P})$, where $\mathcal{P}_\varepsilon(\mathcal{P})$ is defined as
\begin{equation}\label{hubermodel}
    \mathcal{P}_\varepsilon(\mathcal{P}) = \{P_\varepsilon = (1-\varepsilon)P+\varepsilon G: \, \varepsilon \in [0,1], \, P \in \mathcal{P}, \, G \in \mathcal{G}\}.
\end{equation}
The class of distributions $\mathcal{P}_\varepsilon(\mathcal{P})$ is known as the Huber $\varepsilon$-contamination model \citep{huber2004robust} in the robust statistics literature.

As for the privacy, formally speaking, a privacy mechanism is a conditional distribution of the privatised data given the raw data, i.e.~$Q(\cdot | x_1,\ldots,x_n)$, $\{x_i\}_{i = 1}^n \subset \mathcal{X}$, when the raw data are $\{X_i\}_{i = 1}^n = \{x_i\}_{i = 1}^n$.  A privacy mechanism is said to be \emph{$\alpha$-differentially private}, if for all possible observations $\{x_i\}_{i = 1}^n$ and $\{x'_i\}_{i = 1}^n$ that differ in at most one coordinate, it holds that 
\begin{equation} \label{Eq:alphaDP}
    \sup_{A} \frac{Q(A | x_1,\ldots,x_n)}{Q(A | x_1',\ldots,x_n')} \leq e^{\alpha},
\end{equation}
where the supremum is taken over all measurable sets $A$. 

For an $\alpha$-differentially private mechanism to satisfy \emph{the local privacy constraint}, the output is of the form $\{Z_i\}_{i = 1}^n \subset \mathcal{Z}$, where $Z_1$ is generated solely based on $X_1$, and for each $i \in \{2, \ldots, n\}$, $n \geq 2$, $Z_i$ is generated based on $X_i$ and $\{Z_j\}_{j = 1}^{i-1}$.  The constraint \eqref{Eq:alphaDP} can therefore be written in terms of the conditional distributions $Q_i$'s that generate $Z_i$'s, i.e.
\begin{equation}\label{alpha-ldp}
    \max_{i=1,\ldots,n} \sup_{A}\sup_{z_1,\dotsc,z_{i-1} \in \mathcal{Z}}\sup_{x,x'\in \mathcal{X}}\frac{Q_i(A|x,z_1,\dotsc,z_{i-1})}{Q_i(A|x',z_1,\dotsc,z_{i-1})} \leq e^\alpha,
\end{equation}
with the convention that $\{z_1, \cdots, z_j\} = \emptyset$ if $j < 1$.  Any conditional distribution $Q$ that satisfies~\eqref{alpha-ldp} is said to be an \emph{$\alpha$-locally differentially private (LDP) privacy mechanism} \citep[see e.g.][]{duchi2018minimax}. We write $\mathcal{Q}_\alpha$ for the set of all $\alpha$-LDP privacy mechanisms.   

From \eqref{alpha-ldp}, we can see that $\alpha>0$ represents the desired level of privacy which is an input to the data analysis -- the larger $\alpha$ is the less protected the raw data are.  In this paper we focus on the high-privacy regime $0 < \alpha \leq 1$, where $\alpha$ may be a function of the sample size~$n$.  As we will discuss in more detail later, we will require that $n \alpha^2$ diverges, as the sample size $n$ grows unbounded.

With both the ingredients in hand, we define the \emph{$\alpha$-LDP minimax risk under contamination}, which is at the centre of the analysis in this paper; that is,
\begin{equation}\label{intro_minimax}
    \mathcal{R}_{n,\alpha}(\theta(\mathcal{P}), \Phi \circ \rho, \varepsilon) = \inf_{Q \in \mathcal{Q}_\alpha} \inf_{\hat{\theta}} \sup_{P_\varepsilon \in \mathcal{P}_\varepsilon(\mathcal{P})} \mathbb{E}_{P_\varepsilon,Q}\left[\Phi \circ \rho\left\{\hat{\theta}, \, \theta(P)\right\}\right],
\end{equation}
where
\begin{itemize}
    \item the population quantity of interest is $\theta(P) \in \Theta$, denoting a functional supported on $\mathcal{P}$;
    \item the bivariate function $\rho$ is a semi-metric on the space $\Theta$ and $\Phi: \mathbb{R}_{+} \rightarrow \mathbb{R}_{+}$ is a non-decreasing function with $\Phi(0) = 0$;
    \item the first infimum is taken over all possible $\alpha$-LDP privacy mechanisms;
    \item the second infimum is over all measurable functions $\hat{\theta} = \hat{\theta}(Z_1,\dotsc,Z_n)$ of the privatised data generated from privacy mechanism $Q$; and
    \item the loss function is in the form of an unconditional expectation, with respect to both the data generating mechanism $P_{\varepsilon}$ and the privacy mechanism $Q$.
\end{itemize}

As detailed in \eqref{intro_minimax}, our goal is to understand the fundamental limits imposed by both the contamination and the privacy constraint.  In \Cref{general_lower} below, we show that such statistical tasks are at least as hard as either only preserving privacy at level $\alpha$ without contamination or only being robust against contamination without preserving privacy. 

\begin{prop}\label{general_lower}
For any $\varepsilon \in [0, 1)$, define the total variation modulus of continuity $\omega(\varepsilon)$ to be 
    \begin{equation}\label{eq-TV-modulus-conti}
        \omega(\varepsilon) = \sup\{\rho(\theta(R_0), \theta(R_1)): \, \mathrm{TV}(R_0, R_1) \leq \varepsilon/(1-\varepsilon), R_0, R_1 \in \mathcal{P}\}.
    \end{equation}
    Define the $\alpha$-LDP minimax risk without contamination to be 
    \[
        \mathcal{R}_{n,\alpha}(\theta(\mathcal{P}), \Phi \circ \rho) = \mathcal{R}_{n,\alpha}(\theta(\mathcal{P}), \Phi \circ \rho, 0),
    \]
    with $\mathcal{R}_{n,\alpha}(\cdot, \cdot, \cdot)$ defined in \eqref{intro_minimax}. For any given $\alpha \in (0,\infty)$, it holds that 
    \[
        \mathcal{R}_{n,\alpha}(\theta(\mathcal{P}), \Phi \circ \rho, \varepsilon) \geq \mathcal{R}_{n,\alpha}(\theta(\mathcal{P}), \Phi \circ \rho) \vee \frac{\Phi(\omega(\varepsilon)/2)}{2}.
    \]
\end{prop}

We remark that Theorem 5.1 in \cite{chen2018robust} provided a general lower bound under the Huber contamination model for non-private data.  \Cref{general_lower} extends it to a general case that accounts for the LDP constraint. A key aspect of the lower bound in \Cref{general_lower} is that the cost of contamination is separated from that of privacy, i.e.~the level of privacy required does not increase the error introduced by the contamination.  In the sequel, we will show that this lower bound is tight in the estimation problems we consider, while a similar decoupling of a slightly different form can be found for our testing problem.  In the examples we consider, the difficulty of the combined problem of robustness and LDP is the difficulty of the harder one of the two individual problems.  We see this disentanglement as a sign of the connection between privacy and robustness. In our examples we have shown that it is possible to find optimal procedures that are simultaneously privacy-preserving and robust. We have not found any problems for which privacy precludes robustness, or vice-versa. 

The intuition behind \Cref{general_lower} is that if the two distributions are indistinguishable on the raw data space $\mathcal{X}$, then no `transformation'~$Q$ can distinguish them either. Note that a similar quantity to $\omega(\varepsilon)$ was used by \cite{donoho1991geometrizing}, where the Hellinger distance was considered instead of the total variance distance, to translate perturbations in the distribution to perturbations in the quantities of interest measured by the chosen loss function - $\rho(\theta(R_0), \theta(R_1))$.

\subsection{Notation}
For $a,b \in \mathbb{R}$, let $a \wedge b = \min(a,b)$, $a\vee b = \max{(a,b)}$ and $a_+ = a \vee 0$. For non-negative real sequences $\{a_n\}_{n \in \mathbb{N}_+}$ and $\{b_n\}_{n \in \mathbb{N}_+}$, $a_n \ll b_n$ denotes that $\lim_{n\rightarrow\infty}a_n/b_n = 0$, $a_n \gg b_n$ denotes that $b_n \ll a_n$, $a_n \lesssim b_n$ denotes the existence of a constant $C>0$ such that $\limsup_{n \rightarrow \infty} a_n/b_n \leq C$, $a_n \gtrsim b_n$ denotes that $b_n \lesssim a_n$ and $a_n \asymp b_n$ denotes that $a_n \lesssim b_n \lesssim a_n$.  Let $\mathbb{N}_+$ denote all positive integers and $\mathbb{R}_{+}$ denote the set of non-negative real numbers.  Let $|S|$ denote the cardinality of a set $S$.  For two distributions $P_a$ and $P_b$, their total variation distance is $\mathrm{TV}(P_a,P_b) = \sup_{S}|P_a(S)-P_b(S)|$, where the supremum is taken over all measurable sets $S$; and their Kullback--Leibler divergence is $\mathrm{KL}(P_a,P_b) = \int \log(\mathrm{d}P_a/\mathrm{d}P_b)\, \mathrm{d}P_a$, if $P_a$ is absolutely continuous with respect to $P_b$. Let $L_2[0,1] = \{f : [0,1] \rightarrow \mathbb{R}: \int_{0}^1 f^2(x) \, \mathrm{d}x < \infty\}$. A random variable $X$ is sub-exponential with parameters $(\tau, b)$ if $\mathbb{E}[\exp\{\lambda(X-\mathbb{E}(X))\}]\leq \exp(\lambda^2\tau^2/2)$, for $|\lambda| \leq 1/b$.

\section{Robust testing under local differential privacy}\label{sec1}

Under the general setup described in \Cref{sec-general-setup}, assuming that $X_1, \dotsc,X_n$ are i.i.d.~random variables generated from $P$ on $\mathcal{X}$, we consider the robust testing problem 
\begin{align}
    & \mathrm{H}_0: P\in \mathcal{P}_\varepsilon(P_0) = \{P_\varepsilon:\, (1-\varepsilon)P_0+\varepsilon G, \, G \in \mathcal{G}\} \nonumber \\
    \mathrm{vs.} \quad & \mathrm{H}_1: P\in \mathcal{P}_\varepsilon(P_1) = \{P_\varepsilon:\, (1-\varepsilon)P_1+\varepsilon G, \, G \in \mathcal{G}\}, \label{testing_p}
\end{align}
where $P_0$ and $P_1$ are two fixed distributions supported on $\mathcal{X}$, and $\mathcal{G}$ is the set of all distributions supported on $\mathcal{X}$.  In this section, we are interested in testing \eqref{testing_p} under an $\alpha$-LDP constraint. For a given $\alpha$-LDP privacy mechanism $Q$, we let $\Phi_Q = \{\phi:\mathcal{Z}^n \rightarrow \{0,1\}\}$ denote the set of all $\{0,1\}$-valued measurable functions of privatised data $\{Z_i\}_{i = 1}^n$ generated via the privacy mechanism $Q$. The $\alpha$-LDP minimax testing risk can be written as 
\begin{equation}\label{eq-def-test-minimax-risk}
\mathcal{R}_{n,\alpha}(\varepsilon) = \inf_{Q \in \mathcal{Q}_\alpha} \inf_{\phi\in \Phi_Q} \bigg\{\sup_{P \in \mathcal{P}_\varepsilon(P_0)}\mathbb{E}_{P,Q}(\phi) + \sup_{P' \in \mathcal{P}_\varepsilon(P_1)}\mathbb{E}_{P',Q}(1-\phi) \bigg\},
\end{equation}
which corresponds to \eqref{intro_minimax} with $\mathcal{P} = \{P_0,P_1\}$ and $\rho$ being the 0-1 loss, i.e.~$\rho = 0$ if $\phi$ returns the correct hypothesis and $\rho = 1$ otherwise. Note that, by considering the trivial test that always rejects $\mathrm{H}_0$, we have $\mathcal{R}_{n,\alpha}(\varepsilon) \leq 1$.

To fully understand the hardness of \eqref{testing_p}, we construct lower and upper bounds on $\mathcal{R}_{n,\alpha}(\varepsilon)$ in Sections~\ref{sec-lb-test} and \ref{sec-ub-test}, respectively. We conclude in \Cref{sec-disc-test} with a discussion of the existing literature. In \Cref{Sec:RDP} we show that very similar results hold when the LDP constraint is relaxed to a general class of local privacy constraints, namely R\'enyi local differential privacy.

\subsection{Lower bound}\label{sec-lb-test}

In \Cref{testing_lowerbound} below, we provide a lower bound on the $\alpha$-LDP minimax testing risk $\mathcal{R}_{n,\alpha}(\varepsilon)$, in terms of the sample size $n$, the privacy constraint $\alpha$, the total variation distance $\mathrm{TV}(P_0, P_1)$ and the Huber contamination proportion $\varepsilon$.  As we will discuss later, this result also serves as an infeasibility result by providing conditions on the testing problem \eqref{testing_p} for the existence of a non-trivial test with a vanishing risk.

\begin{prop} \label{testing_lowerbound}
For $\alpha \in (0,1)$ and the robust testing problem defined in \eqref{testing_p}, it holds that the $\alpha$-LDP minimax testing risk defined in \eqref{eq-def-test-minimax-risk} satisfies
\[
    \mathcal{R}_{n,\alpha}(\varepsilon) \geq \exp\left\{-16\alpha^2n \{\mathrm{TV}(P_0,P_1)-\varepsilon/(1-\varepsilon)\}_+^2\right\}.
\]
\end{prop}

An immediate consequence of this lower bound is that $\mathcal{R}_{n,\alpha}(\varepsilon) = 1$ whenever $\mathrm{TV}(P_0,P_1) \leq \varepsilon/(1-\varepsilon)$. Indeed, this can be seen by a simpler argument since, in this case, according to Lemma \ref{2eps} there exists some $\widetilde{P} \in \mathcal{P}_\varepsilon(P_0) \cap \mathcal{P}_\varepsilon(P_1)$, and therefore 
\[
\mathcal{R}_{n,\alpha}(\varepsilon) \geq \inf_{Q \in \mathcal{Q}_\alpha} \inf_{\phi\in \Phi_Q} \left\{\mathbb{E}_{\tilde{P},Q}[\phi]+\mathbb{E}_{\tilde{P},Q}[1-\phi] \right\} = 1.
\]
In particular, whenever $\varepsilon \geq 1/2$ we have that $\varepsilon/(1-\varepsilon) \geq 1 \geq \mathrm{TV}(P_0,P_1)$, so that $\mathcal{R}_{n,\alpha}(\varepsilon) = 1$.

\subsection{Upper bound}\label{sec-ub-test}

Given the lower bound in \Cref{testing_lowerbound}, we are to show that a folklore test based on the randomised response mechanism and the Scheff\'{e} set \citep[e.g.][]{devroye2001combinatorial} is minimax rate-optimal for the testing problem \eqref{testing_p}.  

\begin{itemize}
    \item [Step 1.] (Privatisation.) Given data $\{X_i\}_{i = 1}^n$, let $Y_i = \mathbbm{1}\{X_i \in A^c\}$, where $A$ is the Scheff\'{e} set of $P_0$ and $P_1$ in \eqref{testing_p}, i.e.~$A = \argmax_{S \subset \mathcal{X}} \{P_0(S) - P_1(S)\}$.  Let the privatised data~$Z_i$'s be obtained via the randomised response mechanism.  To be specific, let $\{U_i\}_{i = 1}^n$ be independent $\mathrm{Unif}[0, 1]$ random variables that are independent of $\{X_i\}_{i = 1}^n$.  For $i \in \{1, \ldots, n\}$, let
    \begin{equation}\label{eq-private-mechanism-test}
        Z_i = \begin{cases}
            Y_i, & U_i \leq e^{\alpha}/(1 + e^{\alpha}), \\
            1 - Y_i, & \mbox{otherwise}.
        \end{cases}
    \end{equation}
    \item [Step 2.] (Test construction.) Let 
    \[
        \widehat{N}_0 = \sum_{i=1}^n \mathbbm{1}\{Z_i = 0\} \quad \text{and} \quad \widetilde{N}_0 = \frac{e^\alpha + 1}{e^\alpha - 1} \left(\widehat{N}_0 - \frac{n}{e^\alpha + 1}\right).
    \]
    The test is then defined as 
    \begin{equation}\label{eq-test-optimal}
        \tilde{\phi} =\mathbbm{1}\left\{|\widetilde{N}_0/n - P_0(A)|>|\widetilde{N}_0/n - P_1(A)|\right\} = \mathbbm{1} \left\{ 2\widetilde{N}_0/n < P_0(A) + P_1(A) \right\}.
    \end{equation}
\end{itemize}

Note that the privacy mechanism defined in Step 1 satisfies the $\alpha$-LDP constraint in \eqref{alpha-ldp} \cite[e.g.][]{gopi2020locally}.  In fact, the test $\tilde{\phi}$ and a non-private counterpart have been used in the non-robust two-point testing problem with LDP constraints (e.g.~Algorithm 6 in \citealt{joseph2019role}; Algorithm 4 in \citealt{gopi2020locally}) and the robust two-point testing problem without LDP constraints \citep[e.g.\ Section 2 in][]{chen2016general}, respectively.  It is known to be optimal in both cases.  The following theorem, combining previous analyses, shows that this test $\tilde{\phi}$ is still optimal under both the privacy constraint and the presence of contamination.  

\begin{thm}\label{testing_upperbound}
For $\alpha \in (0,1)$ and the robust testing problem defined in \eqref{testing_p}, assuming that $\mathrm{TV}(P_0,P_1) >2\varepsilon$ with $\varepsilon \in [0, 1/2)$, the test defined in \eqref{eq-private-mechanism-test} and \eqref{eq-test-optimal} satisfies that 
\[
    \sup_{P \in \mathcal{P}_\varepsilon(P_0)} \mathbb{E}_{P,Q}(\tilde{\phi}) + \sup_{P' \in \mathcal{P}_\varepsilon(P_1)} \mathbb{E}_{P',Q}(1-\tilde{\phi}) \leq 2\exp[-C\alpha^2n\{\mathrm{TV}(P_0,P_1)-2\varepsilon\}^2],
\]
where $C >0$ is some absolute constant.
\end{thm}

We first note that in \Cref{testing_upperbound}, we require $\varepsilon < 1/2$, which, as discussed above, is necessary for the existence of non-trivial tests.  Further, since $\varepsilon/(1-\varepsilon) \geq \varepsilon$, the lower bound in \Cref{testing_lowerbound} implies that 
\begin{align}
    \mathcal{R}_{n,\alpha}(\varepsilon) \geq \exp\left\{-16\alpha^2n  \{\mathrm{TV}(P_0,P_1)-\varepsilon\}_+^2\right\}. \label{eq-lb-constant-2}
\end{align} 
Comparing the upper bound in \Cref{testing_upperbound} and the lower bound in \eqref{eq-lb-constant-2}, up to constants, we see that the test $\tilde{\phi}$ is optimal in terms of the privacy constraint $\alpha$, the sample size $n$, the separation $\mathrm{TV}(P_0,P_1)$ and the contamination proportion $\varepsilon$.  

\begin{remark}[When $\varepsilon$ is known]
\label{Rmk:TighterTesting}
When $\varepsilon$ is known, a modification of the test procedure defined in \eqref{eq-private-mechanism-test} and \eqref{eq-test-optimal} achieves slightly better performance and shows that consistent testing is possible if and only if $\mathrm{TV}(P_0,P_1) > \varepsilon/(1-\varepsilon)$. With the same privacy mechanism as in \eqref{eq-private-mechanism-test}, consider
\[
    \phi' = \mathbbm{1} \left\{ 2\widetilde{N}_0/n < (1-\varepsilon)\{P_0(A) + P_1(A)\} + \varepsilon \right\},
\]
which uses the same test statistic but compares it to a different critical value. Similar calculations to those carried out for Theorem~\ref{testing_upperbound} show that
\[
    \sup_{P \in \mathcal{P}_\varepsilon(P_0)} \mathbb{E}_{P,Q}(\phi') + \sup_{P' \in \mathcal{P}_\varepsilon(P_1)} \mathbb{E}_{P',Q}(1-\phi') \leq 2\exp[-C'\alpha^2n\{\mathrm{TV}(P_0,P_1)-\varepsilon/(1-\varepsilon)\}_+^2]
\]
for some absolute constant $C' > 0$, provided $\alpha \in (0,1]$. We now see that, when $\mathrm{TV}(P_0,P_1) \leq \varepsilon/(1-\varepsilon)$ we have $\mathcal{R}_{n,\alpha}(\varepsilon)=1$, and when $\mathrm{TV}(P_0,P_1) > \varepsilon/(1-\varepsilon) + 1/\sqrt{n\alpha^2}$ we have
\[
    - \log \mathcal{R}_{n,\alpha}(\varepsilon) \asymp \alpha^2 n \{\mathrm{TV}(P_0,P_1) - \varepsilon/(1-\varepsilon)\}_+^2.
\]
Details of the calculations are given at the end of Section~\ref{appendix:testing}.
\end{remark}

\begin{remark}[Relation to \Cref{general_lower}]
The derived minimax error rate does not match the form given in Proposition~\ref{general_lower}. This form appears to be most suitable for estimation problems, while for testing problems we must look at the error slightly differently. Indeed, it is common in the theory of hypothesis testing to look at conditions under which errors are below given thresholds, often through minimal separation rates, rather than errors themselves \cite[e.g.][]{ingster2003nonparametric}. In problem~\eqref{testing_p} the total variation modulus of continuity is given by
\[
    \omega(\varepsilon) = \sup\bigl\{ \mathbbm{1}_{\{\theta(R_1) \neq \theta(R_0)\}} : \mathrm{TV}(R_0,R_1) \leq \varepsilon/(1-\varepsilon), R_0, R_1 \in \{P_0,P_1\} \bigr\} = \mathbbm{1}_{\{ \mathrm{TV}(P_0,P_1) \leq \varepsilon/(1-\varepsilon) \}},
\]
so we have $\omega(\varepsilon) \leq 0.1$ if and only if $\mathrm{TV}(P_0,P_1) > \varepsilon/(1-\varepsilon)$. Moreover, we have seen that, up to constants, $\mathcal{R}_{n,\alpha}(0) \leq 0.1$ if and only if $\mathrm{TV}(P_0,P_1)$ is larger than $(n \alpha^2)^{-1/2}$ and, on the other hand, $\mathcal{R}_{n,\alpha}(\varepsilon) \leq 0.1$ if and only if $\mathrm{TV}(P_0,P_1) \geq \varepsilon/(1-\varepsilon) + (n\alpha^2)^{-1/2}$. Thus $\mathrm{TV}(P_0,P_1)$ is seen to characterise when each of $\omega(\varepsilon),\mathcal{R}_{n,\alpha}(0),\mathcal{R}_{n,\alpha}(\varepsilon)$ is below the (arbitrary) threshold $0.1$. The level that $\mathrm{TV}(P_0,P_1)$ must exceed for $\mathcal{R}_{n,\alpha}(\varepsilon) \leq 0.1$ is given by the \emph{sum} of the levels required for $\omega(\varepsilon) \leq 0.1$ and $\mathcal{R}_{n,\alpha}(0) \leq 0.1$, and we see a decoupling of the contamination and the privacy.
\end{remark}

\subsection{Discussion}\label{sec-disc-test}

Our results in Proposition \ref{testing_lowerbound} and \Cref{testing_upperbound} share similar spirits as Theorem 5.7 in \cite{joseph2019role}, which stated a minimax lower bound result in terms of the sample complexity and presented an algorithm using a Laplace mechanism that achieves the optimal sample complexity. In particular, they considered a compound hypothesis testing problem where $\mathrm{H}_0$ and $\mathrm{H}_1$ correspond to convex and compact sets of discrete distributions well-separated in terms of total variation distance, whereas our result concerns the Huber contamination model, with different proof schemes.

An extension of the two-point testing problem defined in \eqref{testing_p} is the hypothesis selection, which is popular in both computer science and statistics literature \citep[e.g.][]{gopi2020locally, bun2019private, yatracos1985rates, devroye2001combinatorial, chen2016general}.  To be specific, for a fixed but unknown distribution $P \in \mathcal{P}$, given a set of $k_0 \in \mathbb{N}_+$ distributions $\mathcal{Q} = \{q_1,\dotsc,q_{k_0}\} \subset \mathcal{G}$, one seeks an element in $\mathcal{Q}$ that is closest to $P$ in the total variation distance. In particular, taking $\mathcal{Q}$ to be a $\delta$-covering set ($\delta > 0$) of $\mathcal{P}$, with $k_0$ being the $\delta$-covering number, this hypothesis selection problem is similar to an estimation problem of the distribution $P$.  Based on this setup, \cite{chen2016general} showed that applying a tournament procedure with non-private counterparts of $\tilde{\phi}$, to a $\delta$-covering set of $\mathcal{P}$, is minimax rate-optimal for estimating $P \in \mathcal{P}_\varepsilon(\mathcal{P})$ in terms of the total variation metric.  Their procedure returns an element $\widehat{P} \in \mathcal{Q}$, with high probability, satisfying that 
\[
\mathrm{TV}(\widehat{P},P) \lesssim \{\sqrt{\log(k_0)/n +\delta^2}\}\vee \varepsilon.
\]
With $\delta = \inf\{\delta_1: \, n \geq \log(k_0)/\delta_1^2\}$, it holds that $\mathrm{TV}(\widehat{P}, P) \lesssim \delta \vee \varepsilon$, and $n \geq \log(k_0)/\delta^2$ is called the sample complexity of the procedure. Based on the same setup, \cite{bun2019private} considered the problem under the central privacy model but without contamination, i.e.\ $P \in \mathcal{P}$, and they developed an algorithm that guarantees $\mathrm{TV}(\widehat{P}, P) \lesssim \delta$ with the sample complexity $n \gtrsim \log(k_0)/\delta^2+\log(k_0)/(\delta \alpha)$.  Note that in both these two problems, $k_0$ appears in the sample complexity through its logarithm $\log(k_0)$. 

However, under local privacy constraints, \cite{gopi2020locally} established a lower bound \citep[cf.\ Theorem 2 in][]{gopi2020locally} on the sample complexity of $n \gtrsim k_0/(\delta^2\alpha^2)$, which shows that the cost of this general estimation method induced by a $\delta$-covering set is exponentially higher in the local privacy setting compared to the central privacy and non-private settings. The exponential gap in the sample complexity directly leads to a suboptimal rate for estimation tasks.  In \Cref{ex-1}, we state an example illustrating the suboptimality of hypothesis selection approaches to estimation problems, even in simple parametric problems, which motivates our proposals of \emph{problem-specific} estimation procedures studied in the rest of this paper.

\begin{exmp}\label{ex-1}
Consider $\{X_i\}_{i = 1}^n$ to be i.i.d.~random variables from $\mathcal{N}(\mu, 1)$, $\mu \in [-D, D]$.  Estimating $\mu$ is a special case of the robust mean estimation problem in \Cref{sec2}. For this problem, we have $k_0 \asymp D/\delta$.  Theorem 2 in \cite{gopi2020locally} implies that the sample complexity of the associated hypothesis selection problem is
\[
n \gtrsim \frac{k_0}{\delta^2\alpha^2} \asymp \frac{D/\delta}{\delta^2\alpha^2} = \frac{D}{\delta^3 \alpha^2}, \quad \mbox{i.e. } \delta \gtrsim \left(\frac{D}{n\alpha^2}\right)^{1/3}.
\]
Since $\mathrm{TV}(\mathcal{N}(\mu,1),\mathcal{N}(\mu',1)) \asymp |\mu - \mu'|$ \citep{devroye2018total}, this implies that any estimator $\hat{\mu}$ obtained based on the $\delta$-covering set would have convergence rate measured by $|\hat{\mu}-\mu|$ bounded below by $\{D/(n\alpha^2)\}^{1/3}$. However, since Gaussian distributions have moments of all orders, we will see that we can attain a near-parametric upper bound by applying \Cref{mean_optimalthm} below with arbitrarily large $k$.
\end{exmp}

As for the questions Q1 and Q2 raised in \Cref{sec-contributions}, in this two-point testing problem, we see that
\begin{itemize}
    \item there exists a procedure optimal against contamination \citep[Section 2 in][]{chen2016general} that can be properly privatised to achieve optimal performance; and 
    \item there exists an $\alpha$-LDP procedure \citep{joseph2019role} that is automatically robust and minimax rate-optimal. 
\end{itemize}

\section{Robust mean estimation under local differential privacy}\label{sec2}

Recalling the general setup in \Cref{sec-general-setup}, in this section we consider distributions supported on the real line, i.e.~$\mathcal{X} = \mathbb{R}$, with finite $k$th ($k > 1$) central moments and possibly diverging expectation, as the sample size grows unbounded.  We are interested in estimating the expectation, i.e.~$\theta(P) = \mathbb{E}_{X \sim P}(X)$.  To be specific, we let 
\begin{equation}\label{eq-P-k-dist-def}
    \mathcal{P} = \mathcal{P}_k = \left\{P: \, \mu = \mathbb{E}_{X \sim P}(X) \in [-D,D], \, \sigma^k = \mathbb{E}_{X \sim P}[|X-\mu|^k] \leq 1 \right\},
\end{equation}
where $k$ is considered fixed but arbitrary and $D \geq 1$ may be a function of the sample size.  

We again assume the data are generated from Huber's contamination model \eqref{hubermodel}.  Let $\{X_i\}_{i = 1}^n$ be i.i.d.~random variables with distribution $P_{k, \varepsilon} \in \mathcal{P}_\varepsilon(\mathcal{P}_k)$ and suppose that we are interested in estimating the expectation of the inlier distribution $\mu$.

The $\alpha$-LDP minimax risk defined in \eqref{intro_minimax} takes its specific form in the robust mean estimation problem as follows
\begin{equation}\label{eq-alpha-ldp-def-mean}
    \mathcal{R}_{n,\alpha}(\varepsilon) = \mathcal{R}_{n,\alpha}(\theta(\mathcal{P}_k), (\cdot)^2, \varepsilon) =  \inf_{Q \in \mathcal{Q}_\alpha} \inf_{\hat{\mu}} \sup_{P_{k, \varepsilon} \in \mathcal{P}_\varepsilon(\mathcal{P}_k)}\mathbb{E}_{P_{k, \varepsilon}, Q}\{(\hat{\mu}-\mu)^2\},
\end{equation}
where the metric of interest is the squared loss and the infimum over $\hat{\mu}$ is taken over all measurable functions of the privatised $\{X_i\}_{i = 1}^n$ via some privacy mechnism $Q \in \mathcal{Q}_\alpha$. 

\subsection{Lower bound}

We first provide in Proposition \ref{mean_lowerbound} a lower bound that comprises three terms. If $\log(D) \geq 32n\alpha^2$, then $\mathcal{R}_{n,\alpha}(\varepsilon) \gtrsim 1$.  This implies that consistent estimation of $\mu$ is impossible. When $\log(D) \ll n\alpha^2$, the lower bound simplifies to $(n\alpha^2)^{1/k-1} \vee \varepsilon^{2-2/k}$, which is the minimax rate of the robust mean estimation problem with fixed $D$. We show in Theorem \ref{mean_optimalthm} that a matching upper bound can be achieved by a novel non-interactive procedure under minimal conditions. 

\begin{prop}\label{mean_lowerbound}
Let $\{X_i\}_{i = 1}^n$ be i.i.d.~random variables from $P_{k, \varepsilon} \in \mathcal{P}_\varepsilon(\mathcal{P}_k)$, with $\mathcal{P}_k$ defined in~\eqref{eq-P-k-dist-def}.  For $\alpha \in (0,1]$ and $n \geq n_0$ with a large enough absolute constant $n_0 \in \mathbb{N}_+$, it holds that the $\alpha$-LDP minimax estimation risk defined in \eqref{eq-alpha-ldp-def-mean} satisfies
\[
    \mathcal{R}_{n,\alpha}(\varepsilon) \gtrsim (n\alpha^2)^{1/k-1} \vee \varepsilon^{2-2/k} \vee \frac{D^2}{\exp{(64n\alpha^2)}}.
\]
\end{prop}

\Cref{mean_lowerbound} explains the hardness of this estimation problem by isolating the cost of preserving privacy, contamination and estimating a possibly diverging mean respectively.  If any of these three becomes too hard, this estimation problem becomes infeasible.  In terms of the preservable level of privacy, we require $\alpha \gg n^{-1/2}$ for consistency.  In terms of the Huber contamination proportion, we require $\varepsilon \ll 1$.  In terms of $D$, the upper bound on the absolute mean, we require $D \ll \exp(32n \alpha^2)$.  

The dependence on $D$ is a rather interesting finding. When $k=2$ and $D=\infty$, it is shown in Appendix G of \cite{duchi2013local} that $\mathcal{R}_{n,\alpha}(0) = \infty$. To interpolate between the case of an unbounded parameter space and the case where the mean takes values in a fixed compact set, we prove in \Cref{tomnotes} that, for any $D \in [0,\infty)$ and $n \in \mathbb{N}_+$, the lower bound
\[
\mathcal{R}_{n,\alpha}(0) \geq \frac{D^2}{32\exp{(64n\alpha^2)}}
\]
holds.

Another interesting aspect roots in the derivation of the term $(n\alpha^2)^{1/k-1}$, which unveils a somewhat deeper connection between locally private estimation and estimation within Huber's contamination model. To derive the term $(n\alpha^2)^{1/k-1}$, we apply Corollary 3.1 in \cite{rohde2020geometrizing}, which crucially depends on the following quantity 
\[
    \omega'(\eta) = \omega(\eta/(1+\eta)) = \sup\{|\mu_0-\mu_1|: \mathrm{TV}(R_0,R_1) \leq \eta, \, R_0,R_1 \in \mathcal{P}_k\}, 
\]
where $\mu_0$ and $\mu_1$ denote the means of $R_0$ and $R_1$ respectively.  It is almost identical to $\omega(\cdot)$ defined in \eqref{eq-TV-modulus-conti}, up to a change of variable, and has been shown to play a central role in establishing minimax rates for locally private estimation problems \citep{rohde2020geometrizing}. Therefore, we see that the total variation modulus of continuity $\omega(\cdot)$, a single unifying quantity, can quantify the costs of both privacy and contamination in the Huber's contamination model.

\subsection{Upper bound}

We are now to show that the lower bound in \Cref{mean_lowerbound} is indeed tight by providing a novel non-interactive estimator that is adaptive to $D$. We split the data into four folds, each of which is privatised separately. The procedure has two main steps.
\begin{itemize}
    \item Firstly, for a pre-specified $M > 0$, the first fold is used to construct a private histogram with bin width $M/3$ and to identify bins containing a proportion of the contaminated distribution exceeding a threshold. This allows us to find a larger bin of width $M$ of the form $[\{S+(L-1)/3\}M, \{S+1+(L-1)/3\}M]$, with $S \in \mathbb{N}_+$ and $L \in \{0,1,2\}$, such that only a negligible proportion of the distribution lies outside this interval. This interval can be thought of as a crude, initial estimate of the location of the distribution.
    \item Secondly, for each $\ell \in \{0,1,2\}$ and each data point $X_i$ in fold $\ell+2$, we divide $X_i - (\ell-1)M/3$ by $M$ and use a Laplace mechanism to privatise the remainder. The privatised data from fold $L+2$ can then be used to pinpoint the mean within the interval $[\{S+(L-1)/3\}M, \{S+1+(L-1)/3\}M]$, by adding the privatised remainders to its left-hand end point.
\end{itemize}
Crucially, the value of $L$ is not needed to privatise the data from the final folds and is only used when constructing the final estimator. This can be done because there are only three possible values of $L$ and we may reserve a fold for each eventuality. The other two folds are discarded.

We now formalise the procedure described above. 

\begin{itemize}
    \item [Step 1.] (Privatisation.)  For some absolute constant $c > 0$, let $T = \exp(c n \alpha^2)$.  Let $M$ be a tuning parameter such that $T/M \in \mathbb{N}$.  Let
    \[
        A_j = [(j-1)M/3, jM/3), \quad j \in \mathcal{J} = \{-3T/M, -3T/M+1, \ldots, 3T/M, 3T/M+1\}.
    \]
    Let data be $\{X_i\}_{i = 1}^{4n}$.  Generate independent standard Laplace variables $(W_{ij})_{i \in [n], j \in \mathcal{J}}$, $(W_i^{(\ell)})_{i=(\ell+1) n +1}^{(\ell+2) n}$ for $\ell=0,1,2$.  For $i \in [n]$ and $j \in \mathcal{J}$ set
    \begin{equation}\label{eq-priv-fold1}
        Z_{ij} = \mathbbm{1}_{\{X_i \in A_j\}} + \frac{2}{\alpha} W_{ij}.
    \end{equation}
    For $\ell=0,1,2$ and $i=(\ell+1) n +1,\ldots, (\ell+2)n$ set
    \begin{align*}
        R_i^{(\ell)} = \min \bigl\{ X_i - (j-1)M/3 : j \in \mathcal{J}, j \equiv \ell \, (\text{mod} 3), X_i \geq (j-1)M/3 \bigr\}
    \end{align*}
    and
    \begin{equation}\label{eq-priv-foldl}
        Z_i^{(\ell)} = [R_i^{(\ell)}]_0^M + \frac{M}{\alpha} W_i^{(\ell)},
    \end{equation}
    where $[\cdot]_{0}^{M} = \min\{\max\{\cdot, 0\}, \, M\}$.
    
    \item [Step 2.] (Estimator construction.) With 
    \[
        \delta=T^{-2}(n\alpha^2)^{-1} \quad \mbox{and} \quad \tau=\epsilon + (1-\epsilon)(6/M)^k + 4\sqrt{2 \log(12T/(M\delta))/(n\alpha^2)},
    \]
    define
    \[
        \widehat{\mathcal{J}} = \biggl\{ j \in \mathcal{J} : \frac{1}{n} \sum_{i=1}^n Z_{ij} \geq \tau \biggr\}. 
    \]
    If $\widehat{\mathcal{J}} = \emptyset$ then output $\hat{\mu}=0$. Otherwise, let
    \[
        J=\max \widehat{\mathcal{J}} -1  \quad \mbox{and} \quad L = \begin{cases}
            0, & \text{if } J \equiv 0 \, (\text{mod} 3), \\
            1, & \text{if } J \equiv 1 \, (\text{mod} 3), \\
            2, & \text{if } J \equiv 2 \, (\text{mod} 3). 
        \end{cases}
    \]
    The estimator is defined as
    \begin{equation}\label{eq-hat-mu-NI-2}
        \hat{\mu} = \frac{1}{n} \sum_{i=(L+1)n+1}^{(L+2)n} Z_i^{(L)} + (J-1)M/3.
    \end{equation}
\end{itemize}

The privatisation in \eqref{eq-priv-fold1} and \eqref{eq-priv-foldl} is non-interactive and satisfies $\alpha$-LDP, as shown in \Cref{lemma-mean-alpha-ldp}.  The property of the estimator $\hat{\mu}$ defined in \eqref{eq-hat-mu-NI-2} is collected in \Cref{mean_optimalthm} below.

\begin{thm}\label{mean_optimalthm}
Given i.i.d.~random variables $\{X_i\}_{i = 1}^{4n}$ with distribution $P_{k,\varepsilon} = (1-\varepsilon)P_k+\varepsilon G$, where $P_k \in \mathcal{P}_k$ defined in \eqref{eq-P-k-dist-def}, and $G$ is any arbitrary distribution supported on~$\mathbb{R}$.  The estimator $\hat{\mu}$ defined in \eqref{eq-hat-mu-NI-2}, with inputs satisfying that (i) $T \geq D$, (ii) $(n\alpha^2)^{-1} \log(12T^3 n \alpha^2/M) \leq \min(\alpha^{-2}/2, 1/512)$ and (iii) $\varepsilon+(1-\varepsilon)(6/M)^k \leq 1/12$, satisfies that 
\[
    \mathbb{E}\{ (\hat{\mu} - \mu)^2 \} \lesssim T^2 \exp(-n \alpha^2/512) + \frac{M^2}{n \alpha^2} + \varepsilon^2 M^2 + M^{-2(k-1)}.
\]
In particular, choosing $T=\exp(n\alpha^2/3072)$ and $M \asymp \varepsilon^{-1/k} \wedge (n \alpha^2)^{1/(2k)}$, when $\alpha \leq 1$, $D\leq T$ and $\min\{\varepsilon,(n\alpha^2)^{-1}\} \leq c_0$ for some small constant $c_0 > 0$, we have that
\[
    \mathbb{E}\{ (\hat{\mu} - \mu)^2 \} \lesssim (n \alpha^2)^{1/k-1} \vee \varepsilon^{2-2/k}.
\]
\end{thm}

With the properly chosen truncation tuning parameter $M$, \Cref{mean_optimalthm} shows that $\hat{\mu}$ gives an upper bound on the mean squared error of $(n\alpha^2)^{1/k-1} \vee \varepsilon^{2-2/k}$. Comparing with \Cref{mean_lowerbound}, we see that $\hat{\mu}$ is minimax rate-optimal in terms of the dependence on $n$, $\alpha$ and $\varepsilon$.  It is notable that it remains rate-optimal even for a diverging $D$, provided that it is not growing faster than a certain exponential rate in $n\alpha^2$ (more precisely, we require $D \leq T = \exp(n\alpha^2/3072)$).  In fact, this requirement is essentially unavoidable, as \Cref{mean_lowerbound} shows that consistent estimation is impossible if $\log D \gg n \alpha^2$. Lastly, \Cref{mean_optimalthm} shows that the error of our estimator has an upper bound independent of $D$, provided $D \leq T$, which is in contrast to the performance of the private mean estimator based on the standard Laplace mechanism studied in \cite{duchi2018minimax} (see \Cref{meanupper1} in \Cref{sec-duchi-mean}).

\subsection{Discussions}\label{sec-disc-mean}

Our focus is on cases where both contamination and privacy constraints are present.  In this section, we discuss some related results in the literature that only consider either contamination or the local privacy constraint. 

Without the local privacy constraint, \cite{prasad2020robust} proposed a two-step robust mean estimator (cf.\ Algorithm 2 therein) that achieves the optimality in the model \eqref{eq-P-k-dist-def} with $k \geq 2$ and $D = \infty$. With probability at least $1-\delta$, their estimator $\hat{\mu}_{\mathrm{Pra}}$ satisfies that 
\begin{equation}\label{eq-nonprivacy-mean-rate}
    |\hat{\mu}_{\mathrm{Pra}} - \mu|^2 \lesssim n^{-1}\log(1/\delta) \vee \varepsilon^{2-2/k},
\end{equation}
which is known to be information-theoretically optimal \citep[e.g.][]{diakonikolas2019robust}.  Recalling that our results are in the form of expectations, to compare \eqref{eq-nonprivacy-mean-rate} with \Cref{mean_optimalthm}, we hence ignore the logarithmic term.  Regarding the term involving $\varepsilon$, we see that in \Cref{mean_optimalthm}, it is completely isolated from the privacy level $\alpha$; and in both \eqref{eq-nonprivacy-mean-rate} and \Cref{mean_optimalthm}, it is of the order $\varepsilon^{2-2/k}$. As for preserving privacy at level $\alpha$, we see the effects are two-fold (see also Section~3.2.1 in \citealt{duchi2013local}): (1) a reduction of the effective sample size from $n$ to $n\alpha^2$; and (2) instead of $n^{-1}$ in the non-private case, we have in \Cref{mean_optimalthm} the term $n^{1/k-1}$, i.e.~the heavy-tailedness of $P\in \mathcal{P}_k$ has no effect on the non-private convergence rate in \eqref{eq-nonprivacy-mean-rate}, whereas a loss of $1/k$ in the exponent is incurred in the private setting. Given that we have shown in \Cref{mean_lowerbound}, the rate we achieved in \Cref{mean_optimalthm} is optimal, the loss in the rate $n^{1/k}$ is unavoidable due to the privacy constraint -- similar phenomena have also been observed in the nonparametric estimation literature \citep[e.g.][]{berrett2019classification, berrett2021changepoint} as well as in \Cref{sec3}. 

In view of Q1 in \Cref{sec-contributions}, we claim that the estimator $\hat{\mu}_{\mathrm{Pra}}$ is similar to our estimator $\hat{\mu}$ in a broad sense. As detailed in Algorithm~2 in \cite{prasad2020robust}, $\hat{\mu}_{\mathrm{Pra}}$ first constructed a shortest interval initial estimator using half of the data, which is similar to finding $J$ using the first $n$ samples in our construction serving as a crude estimate. The remaining data are then used to refine this crude estimate. One may, therefore, see that our estimator $\hat{\mu}$ as a non-interactive and private version of an optimal robust mean estimator.

In view of Q2 in \Cref{sec-contributions}, \cite{duchi2018minimax} studied the problem \eqref{eq-P-k-dist-def} with $D = 1$ and $\varepsilon = 0$ (cf.~Section 3.2.1 therein). They estimated the mean by averaging the privatised data obtained by adding Laplace noise to the truncated data. This mechanism guarantees the $\alpha$-LDP constraint, but in the case when $D$ is large and $\varepsilon>0$, it is sub-optimal.  See \Cref{meanupper1}.  

To summarise, in this robust mean estimation problem, we see that 
\begin{itemize}
    \item there exists a procedure \citep{prasad2020robust} optimal against contamination that can be properly privatised to achieve optimal performance; and 
    \item a standard $\alpha$-LDP procedure is not automatically robust and minimax rate-optimal when the parameter space grows (cf.~\Cref{meanupper1}). 
\end{itemize}

\section{Robust density estimation under local differential privacy}\label{sec3}

Recalling the general setup in \Cref{sec-general-setup}, in this section we consider distributions supported on $\mathcal{X} = [0, 1]$, belonging to the Sobolev class $\mathcal{P} = \mathcal{F}_{\beta}$, defined below.

\begin{defn}[Sobolev class] \label{sobolev}
Let $\beta > 1/2$ and $\{\gamma_t\}_{t\in \mathcal{T}}$ be an orthonormal basis of $L_2[0, 1]$ indexed by a countable family $\mathcal{T}$.  For a given coefficient sequence $\{a_t\}_{t \in \mathcal{T}}$ associated with $\{\gamma_t\}_{t\in \mathcal{T}}$, the Sobolev class $\mathcal{F}_{\beta}$ is defined as
\begin{align} \label{generalelip}
    &\mathcal{F}_{\beta} =\bigg\{f:\, [0, 1] \rightarrow \mathbb{R}_{+}\; \Big|\; \int_{[0, 1]} f(x) \, \mathrm{d}x = 1, \nonumber \\
    & \hspace{3cm} \sum_{t \in \mathcal{T}}|a_t|^{2\beta} \left|\int_{[0, 1]} f(x)\gamma_t(x)\,\mathrm{d}x\right|^2 = \sum_{t \in \mathcal{T}}|a_t|^{2\beta}|f_t|^2 < \infty\bigg\}.
\end{align}
\end{defn}

We again assume that the data are generated from Huber's contamination model \eqref{hubermodel}.  When $\{X_i\}_{i = 1}^n$ are i.i.d.~random variables with distribution $P_{f_\varepsilon} \in \mathcal{P}_\varepsilon(\mathcal{F}_\beta)$, we are interested in estimating the density of the inlier distribution $\theta(P) = f$.

The $\alpha$-LDP minimax risk defined in \eqref{intro_minimax} takes its specific form in the robust density estimation problem as follows
\begin{equation}\label{eq-minimax-risk-def-density}
    \mathcal{R}_{n,\alpha}(\mathcal{F}_\beta, \Phi \circ \rho, \varepsilon) =  \inf_{Q \in \mathcal{Q}_\alpha} \inf_{\tilde{f}} \sup_{P_{f_\varepsilon} \in \mathcal{P}_\varepsilon (\mathcal{F}_\beta)} \mathbb{E}_{P_{f_\varepsilon},Q}\Big\{\Phi \circ \rho (\tilde{f},f)\Big\},
\end{equation}
where the infimum over $\tilde{f}$ is taken over all measurable functions of the privatised data generated from some $Q \in \mathcal{Q}_\alpha$.  We consider two loss functions, both of which are commonly used in the nonparametric estimation literature \citep[e.g.][]{tsybakov2009introduction}.  To be specific, we consider $\Phi \circ \rho$ to be the squared-$L_2$ loss and the $L_\infty$ loss, separately.  For any two density functions $g_1$ and $g_2$ supported on $[0, 1]$, let the squared-$L_2$-loss and the $L_\infty$ loss be
\begin{equation}\label{eq-losses-def-density}
    \|g_1 - g_2\|_2^2 = \int_{[0, 1]} \{g_1(x) - g_2(x)\}^2 \,\mathrm{d}x  \quad \text{and} \quad \|g_1 - g_2\|_\infty = \sup_{x\in[0,1]} |g_1(x) - g_2(x)|, 
\end{equation}
respectively.

\subsection{Lower bound}\label{npd_lowerbound} 

In \Cref{npd_lowerbound_prop} below, we present lower bounds on $\mathcal{R}_{n,\alpha}(\mathcal{F}_\beta, \Phi \circ \rho, \varepsilon)$, with $\Phi \circ \rho$ taken to be the squared-$L_2$ loss and $L_{\infty}$ loss.  As we shall discuss later, \Cref{npd_lowerbound_prop} is an application of the general lower bound result \Cref{general_lower}.

\begin{prop}\label{npd_lowerbound_prop}
Let $\{X_i\}_{i = 1}^n$ be i.i.d.~random variables with distribution $P_{f_\varepsilon} \in \mathcal{P}_\varepsilon(\mathcal{F}_\beta)$.  For $\alpha \in (0,1)$, it holds that the $\alpha$-LDP minimax estimation risks defined in \eqref{eq-minimax-risk-def-density}, equipped with the squared-$L_2$ loss and the $L_\infty$ loss defined in \eqref{eq-losses-def-density}, are
\begin{equation}\label{npd_lowerbound_2}
    \mathcal{R}_{n,\alpha}(\mathcal{F}_\beta,\|\cdot\|^2_2,\varepsilon) \gtrsim (n\alpha^2)^{-\frac{2\beta}{2\beta+2}} \vee  \varepsilon^{\frac{4\beta}{2\beta+1}}
\end{equation}
and
\begin{equation}\label{npd_lowerbound_inf}
   \mathcal{R}_{n,\alpha}(\mathcal{F}_\beta,\|\cdot\|_\infty,\varepsilon) \gtrsim \left\{\frac{\log(n\alpha^2)}{n\alpha^2}\right\}^{\frac{2\beta-1}{4\beta+2}} \vee \varepsilon^{\frac{2\beta-1}{2\beta+1}},
\end{equation}
respectively.
\end{prop}

In view of \Cref{general_lower}, the results in \Cref{npd_lowerbound_prop} are obtained by separately controlling the lower bounds on (i) the non-robust cases $\mathcal{R}_{n,\alpha}(\mathcal{F}_\beta,\|\cdot\|^2_2,0)$ and $\mathcal{R}_{n,\alpha}(\mathcal{F}_\beta,\|\cdot\|_\infty,0)$ and (ii) the total variation modulus of continuity $\omega(\varepsilon)$ under different loss functions.  

As for (i), due to the Sobolev embedding theorem for Besov space \citep[e.g. Proposition 4.3.20 in][]{gine2021mathematical}, it can be seen as a special case of Corollary 2.1 in \cite{butucea2020local}, which is a result on minimax rates for estimating density functions under $L_r$ loss ($r \geq 1$) without contamination but with an $\alpha$-LDP constraint.

As for (ii), we construct a lower bound on $\omega(\varepsilon)$ by considering a pair of density functions with the help of wavelet basis functions.  We note that Theorem 3 in \cite{uppal2020robust} studied a robust density estimation problem in the Besov space, under the Besov integral probability metrics but without privacy constraints.    

\subsection{Upper bounds}\label{npd_upperbound}

In this subsection, we study the robustness property of two $\alpha$-LDP estimators of the density function, subject to the squared-$L_2$ loss and the $L_{\infty}$ loss, in Theorems~\ref{L2_risk} and \ref{infinity_risk}, respectively.  Both estimators are in the form of linear projection estimators, but constructed based on different choices of orthonormal basis functions and privacy mechanisms.  

\medskip
\noindent \textbf{The squared-$L_2$ loss.}  To obtain an upper bound on $\mathcal{R}_{n,\alpha}(\mathcal{F}_\beta,\|\cdot\|^2_2,\varepsilon)$, we analyse a projection estimator based on the orthonormal basis for $L^2[0,1]$ consisting of trigonometric functions, i.e.
\begin{equation}\label{trigonometricbasis}
    \varphi_1(x) = 1, \quad \varphi_{2j}(x) = \sqrt{2}\cos(2\pi jx) \quad \mbox{and} \quad \varphi_{2j+1}(x) =\sqrt{2}\sin(2\pi jx), \quad j \in \mathbb{N}_+.
\end{equation}
With the choices $a_j = \lfloor (j+1)/2\rfloor$ in \eqref{generalelip}, the Sobolev space in \Cref{sobolev} is characterised by 
\begin{equation}\label{ellipsoid}
    \sum_{j=1}^\infty j^{2\beta}\theta_j^2 = \sum_{j=1}^\infty j^{2\beta}\left\{\int_{0}^1 f(x) \varphi_j(x)\,\mathrm{d}x\right\}^2 \leq r^2 < \infty,
\end{equation}
where $r > 0$ is some universal constant controlling the radius of the ellipsoid.  Condition \eqref{ellipsoid} can be translated into smoothness conditions on functions having $\beta$th order (weak) derivative in $L^2[0,1]$ \citep[e.g. Proposition 1.14 in][]{tsybakov2009introduction}.  Our projection estimator $\hat{f}$ is then constructed as follows.

\begin{itemize}
    \item [Step 1.] (Privatisation.) Given data $\{X_i\}_{i = 1}^n$, for any $i \in \{1, \ldots, n\}$, let $v = [\varphi_j(X_i)]_{j=1}^k \in \mathbb{R}^k$, where $k \in \mathbb{N}_+$ is a pre-specified tuning parameter.  Let the corresponding privatised data be $Z_i \in \mathbb{R}^k$ satisfying that $\mathbb{E}_Q[Z_{i,j}|X_i] = \varphi_j(X_i)$ and $Z_{i,j} = \{-B,B\}$ with $B \lesssim \sqrt{k}/\alpha$, where $\mathbb{E}_Q$ denotes the expectation under the privacy mechanism.  See \Cref{addition4} for full details of this privacy mechanism. 
    \item [Step 2.] (Estimator construction.) Let 
    \begin{equation}\label{fhat}
        \hat{f} = \sum_{j=1}^k \overline{Z}_j \varphi_j = \sum_{j=1}^k \left(\frac{1}{n}\sum_{i=1}^n Z_{i,j}\right)\varphi_j.
    \end{equation}
\end{itemize}

In the construction of $\hat{f}$, we choose the trigonometric functions, which satisfy the bounded basis function condition with 
\begin{equation}\label{eq-bounded-basis-fourier}
    \max_{1 \leq j \leq k}\|\varphi_j(\cdot)\|_{\infty} \leq \sqrt{2}.
\end{equation}
  We remark that one may also choose other basis functions, provided that all basis functions are upper bounded by an absolute constant in the function supremum norm.  This is to guarantee that the $\alpha$-LDP constraint (\ref{alpha-ldp}) holds \citep{duchi2018minimax}.  Another input is the truncation number $k$, which essentially truncates this infinite-dimensional nonparametric problem to a $k$-dimensional estimation problem.  In fact, the estimator \eqref{fhat} is considered in Section 5.2.2 of \citet{duchi2018minimax}, as a nonparametric density estimator, which is shown to be minimax rate optimal in terms of the squared-$L_2$ norm under local differential privacy without contamination \citep[cf.\ Corollary 7 in][]{duchi2018minimax}.  Together with \Cref{L2_risk} below, we will show that this private procedure is automatically robust, in view of Q2 in \Cref{sec-contributions}.

We remark that $\hat{f}$ defined in \eqref{fhat} is a privatised version of the widely-used nonparametric projection estimator \citep[e.g.\ Chapter 1 in][]{tsybakov2009introduction}, defined as
\[
\tilde{f} = \sum_{j=1}^k \left\{\frac{1}{n}\sum_{i=1}^n\varphi_j(X_i)\right\}\varphi_j, 
\]
which attains the optimal minimax rate under squared-$L_2$ loss when applied to nonparametric regression problems for functions belonging to the Sobolev class \eqref{ellipsoid} \citep[cf.\ Theorem 1.9 in][]{tsybakov2009introduction}.

\begin{thm}[The squared-$L_2$ norm case]\label{L2_risk}
Given i.i.d.~random variables $\{X_i\}_{i = 1}^n$ with distribution $ P_{f_\varepsilon} =  (1-\varepsilon)P_f+\varepsilon G$, where $P_f$ denotes the distribution with the density function $f \in \mathcal{F}_{\beta}$ defined in \eqref{generalelip}, and $G$ is an arbitrary distribution supported on $[0,1]$, for $\alpha \in (0,1)$ and $\varepsilon \in [0,{1}]$, the estimator $\hat{f}$ defined in \eqref{fhat} satisfies that
\[
    \mathbb{E}_{P_{f_\varepsilon},Q}[\|\hat{f} - f\|_2^2] \lesssim \varepsilon^{\frac{4\beta}{2\beta+1}} \vee (n\alpha^2)^{-\frac{2\beta}{2\beta+2}},
\]
when the tuning parameter $k$ is chosen to be $k \asymp \varepsilon^{-\frac{2}{2\beta+1}} \wedge (n\alpha^2)^{\frac{1}{2\beta +2}}$.
\end{thm}

Comparing with the lower bound in \eqref{npd_lowerbound_2}, \Cref{L2_risk} shows that under a suitable choice of~$k$, the estimator $\hat{f}$ is minimax rate-optimal in the presence of contamination and a local privacy constraint.   

\medskip
\noindent \textbf{The $L_{\infty}$ loss.}  Our next goal is to obtain an upper bound on $ \mathcal{R}_{n,\alpha}(\mathcal{F}_\beta,\|\cdot\|_\infty,\varepsilon) $, and it turns out that the trigonometric basis used in the previous analysis under the squared-$L_2$ risk is not flexible enough to obtain the optimal rate in the $L_\infty$ case. One reason is that it prevents the use of Laplace mechanism to construct optimal private estimator. As discussed in the last paragraph of Section 5.2.2 in \cite{duchi2018minimax}, even under $L_2$ loss, adding Laplace noise to the trigonometric basis leads to sub-optimal rate of order $(n\alpha^2)^{-2\beta/(2\beta+3)}$, which is worse than the corresponding optimal rate $(n\alpha^2)^{-2\beta/(2\beta+2)}$ as that in Theorem \ref{L2_risk}.  We instead exploit the wavelet basis, which allows the use of a simple Laplace mechanism to efficiently privatise the empirical wavelet coefficients \citep{butucea2020local}.  

Given a father wavelet $\phi: [-A, A] \to \mathbb{R}$ and a mother wavelet $\psi: [-A, A] \to \mathbb{R}$, where $A > 0$ is an absolute constant (see e.g.~Section~4.2.1 in \citealt{gine2021mathematical}), satisfying 
\begin{equation}\label{regular}
    \int_{\mathbb{R}} \psi(x)\,\mathrm{d}x = 0,  \quad \|\psi\|_{\infty} <\infty \quad \text{and} \quad \|\phi\|_{\infty} <\infty,
\end{equation}
a wavelet basis of $L^2(\mathbb{R})$ can be formed as
\[
\{\phi_k = \phi(\cdot - k): k \in \mathbb{Z}\} \cup \{\psi_{jk} = 2^{j/2}\psi(2^j(\cdot)-k): j \in \mathbb{N}_+\cup \{0\}, k \in \mathbb{Z}\}.
\]
We denote $\phi_k$ as $\psi_{-1k}$ to simplify the notation. Given such a basis, we have that for any $f \in L^2(\mathbb{R})$,
\begin{equation}\label{eq-f-decompose-wavelet}
    f(x) =  \sum_{j \geq -1} \sum_{k \in \mathbb{Z}}\beta_{jk}\psi_{jk}(x) =  \sum_{j \geq -1} \sum_{k \in \mathbb{Z}} \left\{\int_{\mathbb{R}} f(x')\psi_{jk}(x')\, \mathrm{d}x'\right\} \psi_{jk}(x).
\end{equation}
Note that one can periodise a wavelet basis of $L^2(\mathbb{R})$ to construct a basis on $L^2[0,1]$ or apply boundary correction to wavelet basis for non-periodic functions supported on $[0,1]$ while regular wavelet properties including \eqref{regular} still hold. We refer to \cite{gine2021mathematical} and \cite{daubechies1992ten} for more detailed introduction on the theory of wavelets.

Since the density functions and wavelet basis are assumed to have compact supports on $[0,1]$ and $[-A,A]$ respectively, the representation \eqref{eq-f-decompose-wavelet} implies that for each resolution level $j \geq -1$, $|\mathcal{N}_j| = \left|\{k: \, \beta_{jk} \neq 0\}\right| \leq 2^j + 2A + 1$. With the choice $a_j = 2^j$ in \eqref{generalelip}, the Sobolev space in Definition \ref{sobolev} can be characterised by
\begin{equation}\label{besov_sobolev}
     \sum_{j\geq -1}(2^j)^{ 2\beta }\|\beta_{j\cdot}\|_2^2 = \sum_{j\geq -1}(2^j)^{ 2\beta } \left(\sum_{k \in \mathcal{N}_j}\beta^2_{jk} \right) < \infty.
\end{equation}
Condition \eqref{besov_sobolev} can also be translated to smoothness conditions on functions having $\beta$th order (weak) derivative in $L^2(\mathbb{R})$ \citep[e.g.\ Proposition 4.3.20 in][]{gine2021mathematical}.

The projection estimator is then constructed as follows.
\begin{itemize}
    \item [Step 1.] (Privatisation.)  Given data $\{X_i\}_{i = 1}^n$, for any $i \in \{1, \dotsc, n\}$, $j \in \{-1, 0, \dotsc, J\}$ and $k \in \mathcal{N}_j$, where $J \in \mathbb{N}_+$ is a pre-specified tuning parameter, let $W_{ijk}$'s be independent standard Laplace random variables which are also independent of $X_i$'s, the noise parameter $\sigma_J$ be
    \begin{equation}\label{sigmaJ}
        \sigma_J = C 2^{J/2}/\alpha, \quad \mbox{with} \quad C = (8\lceil A\rceil+4)\|\psi\|_\infty \sqrt{2}(\sqrt{2}-1)^{-1},
    \end{equation}
    and the privatised empirical wavelet coefficients be
    \[
        \hat{\beta}_{jk} = \frac{1}{n}\sum_{i=1}^n Z_{ijk} = \frac{1}{n} \sum_{i = 1}^n \left\{ \psi_{jk}(X_i)+\sigma_JW_{ijk}\right\}.
    \]
    \item [Step 2.] (Estimator construction.) Let the final estimator be
    \begin{equation}\label{flap}
        \hat{f}_{\mathrm{Lap}} = \sum_{j = -1}^J\sum_{k \in \mathcal{N}_j}\hat{\beta}_{jk}\psi_{jk}.
    \end{equation} 
\end{itemize}

 The tuning parameter $J$ serves as a truncation parameter, reducing an infinite-dimensional nonparametric estimation problem to a finite-dimensional problem with dimensionality being $\sum_{j = -1}^J |\mathcal{N}_j|$.  The noise level $\sigma_J$ is chosen to guarantee the $\alpha$-LDP constraint \citep[Proposition 3.1 in][]{butucea2020local}.

The estimator \eqref{flap} is previously studied in \cite{butucea2020local}, without the presence of contamination but  shown to be optimal in terms of estimating the $L_{\infty}$-loss, under $\alpha$-LDP constraint. We remark that \cite{butucea2020local} studies a more general space and a wider range of loss functions.   Together with \Cref{infinity_risk} below, we will show that this private procedure is also automatically robust, again in view of Q2 in \Cref{sec-contributions}.

\begin{thm}[The $L_\infty$ norm case]\label{infinity_risk}
Given i.i.d.~random variables $\{X_i\}_{i = 1}^n$ with distribution $ P_{f_\varepsilon} =  (1-\varepsilon)P_f+\varepsilon G$, where $P_f$ denotes the distribution with the density function $f \in \mathcal{F}_{\beta}$, defined in \eqref{generalelip}, and $G$ is an arbitrary distribution supported on $[0,1]$, for $\alpha \in (0,1)$ and $\varepsilon \in [0,{1}]$, the estimator $\hat{f}_{\mathrm{Lap}}$ defined in \eqref{flap} satisfies that
\[
    \mathbb{E}_{P_{f_\varepsilon},Q}(\|\hat{f}_{\mathrm{Lap}} - f\|_{\infty}) \lesssim \left\{\frac{\log(n\alpha^2)}{n\alpha^2}\right\}^{\frac{2\beta-1}{4\beta+2}} \vee \varepsilon^{\frac{2\beta-1}{2\beta+1}},
\]
with the tuning parameter $J$ satisfying
    \[
        2^J \asymp \left\{\frac{\log(n\alpha^2)}{n\alpha^2}\right\}^{-\frac{1}{2\beta+1}} \wedge \varepsilon^{-\frac{2}{2\beta+1}}.
    \]
\end{thm}

Comparing with the lower bound in \eqref{npd_lowerbound_inf}, \Cref{infinity_risk} shows that under a suitable choice of $J$, the estimator $\hat{f}_{\mathrm{Lap}}$ is minimax rate optimal. 

\subsection{Discussion}\label{sec-disc-density}

In the classical nonparametric density estimation literature \citep[e.g.][]{tsybakov2009introduction}, it is known that without the presence of contamination or privacy constraints, 
\begin{equation}\label{eq-classical-rates-density}
\mathcal{R}_{n,\infty}(\mathcal{F}_\beta,\|\cdot\|^2_2,0) \asymp n^{-\frac{2\beta}{2\beta+1}} \quad \text{and} \quad \mathcal{R}_{n,\infty}(\mathcal{F}_\beta,\|\cdot\|_\infty,0) \asymp \left\{\frac{\log(n)}{n}\right\}^{\frac{2\beta}{4\beta+2}}. 
\end{equation}
With the presence of both contamination and privacy constraints, in this section, we have shown that 
\[
\mathcal{R}_{n,\alpha}(\mathcal{F}_\beta,\|\cdot\|^2_2,\varepsilon) \asymp (n\alpha^2)^{-\frac{2\beta}{2\beta+2}} \vee  \varepsilon^{\frac{4\beta}{2\beta+1}}   \,\, \text{and} \,\, \mathcal{R}_{n,\alpha}(\mathcal{F}_\beta,\|\cdot\|_\infty,\varepsilon) \asymp \left\{\frac{\log(n\alpha^2)}{n\alpha^2}\right\}^{\frac{2\beta-1}{4\beta+2}} \vee \varepsilon^{\frac{2\beta-1}{2\beta+1}}.
\]
Comparing our results with the classical rates in \eqref{eq-classical-rates-density}, we see that, similar to the mean estimation problem studied in \Cref{sec2}, the cost of preserving privacy is manifested through a reduction of the effective sample size from $n$ to $n\alpha^2$ and a loss in the exponent of convergence rate, both of which have been observed in the literature \citep{duchi2013local,duchi2018minimax,butucea2020local}.  It is, however, interesting to observe that the privacy constraint and the contamination proportion are completely isolated in terms of the fundamental limits. We also remark that the condition of bounded basis function \eqref{eq-bounded-basis-fourier} and \eqref{regular} are critical for both privatising the data and being robust to contamination.

On the other hand -- with contamination but without LDP constraints -- \cite{uppal2019nonparametric} studied a non-private version of $\hat{f}_\mathrm{Lap}$, which is linear, and a non-linear wavelet thresholding estimator for robust estimation of densities belonging to Besov space. They showed that the wavelet thresholding estimator is optimal for a wide range of loss functions, but for the squared-$L_2$ loss and $L_\infty$ loss we considered here, it suffices to use linear estimators to achieve optimality in the presence of contamination. (Similar phenomena were observed in \cite{butucea2020local} when estimating functions in Besov space under local privacy constraint.) Therefore, we may again view $\hat{f}_\mathrm{Lap}$ as a properly privatised version of a robust estimator that yields optimal performance, the success of which depends crucially on the choice of basis function.

In the robust statistics literature \citep[e.g.][]{chen2016general,chen2018robust,uppal2020robust}, an interesting quantity to investigate is the maximum proportion of contamination such that $\mathcal{R}_{n,\alpha}(\theta(\mathcal{P}),\Phi \circ \rho,\varepsilon) \asymp \mathcal{R}_{n,\alpha}(\theta(\mathcal{P}),\Phi \circ \rho,0) $, that is the maximum proportion of contamination that the estimators can tolerate to obtain the optimal rate without contamination.  In the private setting, denoting this quantity as $\varepsilon^*_{\alpha}$, we have that $\varepsilon^*_{\alpha} \asymp (n\alpha^2)^{-(2\beta+1)/(4\beta+4)}$ in the squared-$L_2$ case and $\varepsilon^*_{\alpha} \asymp \left\{\log(n\alpha^2)/(n\alpha^2)\right\}^{1/2}$ in the $L_\infty$ case. Comparing to the non-private setting, denoting this quantity as $\varepsilon^*$, where $\varepsilon^* \asymp n^{-1/2}$ in the squared-$L_2$ case and $\varepsilon^* \asymp \left\{\log(n\alpha^2)/(n\alpha^2)\right\}^{\beta/(2\beta-1)}$ in the $L_\infty$ case, we see that private algorithms can tolerate more contamination but at the price of converging at a slower rate, due to the presence of privacy constraints.

We conjecture that the condition that the density function of interest is supported on a compact set is critical in terms of achieving optimality under both contamination and LDP constraints. 
  As a consequence of this compact support condition, we require bounded basis functions -- see \eqref{eq-bounded-basis-fourier} and \eqref{regular} -- which facilitate our proofs.  Another direct consequence of this compact support condition is the following summary. 

As for the questions Q1 and Q2 raised in \Cref{sec-contributions}, in this robust density estimation problem, we see that
\begin{itemize}
    \item there exist procedures optimal against contamination \citep{uppal2019nonparametric} that can be properly privatised to achieve optimal performance; and 
    \item there are existing $\alpha$-LDP procedures \citep{duchi2018minimax,butucea2020local} that are automatically robust and minimax rate optimal. 
\end{itemize}
\section{Conclusions}

In this paper, we studied various statistical problems under both Huber's $\varepsilon$-contamination model \eqref{hubermodel} and LDP constraints \eqref{alpha-ldp}.  For the four problems concerned in this paper (with one left in \Cref{secE}), we made an attempt to answer Q1 and Q2 in \Cref{sec-contributions}; that is, being aware of the deep connections between robustness and LDP, what we can say about the ability of robust procedures to work with privatised data and about the robustness of private procedures.  For all four problems that we studied, we find procedures that are simultaneously robust, privacy-preserving and statistically rate-optimal.  We commented on the connections between our methods to those which are used only under contamination or only under LDP constraints, and provided partial answers to these two questions in specific cases.  

The optimality of our procedures mostly relies on the knowledge of $\varepsilon$ - an upper bound on the contamination level.  This is an assumption commonly used in the literature \citep[e.g.][]{huber1992robust,lai2016agnostic,prasad2020robust,lugosi2021robust} though impractical, since it is impossible to estimate $\varepsilon$ when the contamination distribution is not specified.  An overly large input of $\varepsilon$ leads to an inflated error bound while a conservative input of $\varepsilon$ leads to unjustified error controls.  There has been some recent work on general approaches to robust methodology when $\varepsilon$ is unknown \citep[e.g.][]{jain2022robust}.  Applied to our procedures, the theoretical guarantees are not immediate since our error bounds hold in expectation rather than almost surely. Nevertheless, it would be interesting to adapt these ideas to further develop private, robust and optimal procedures, which are also adaptive to $\varepsilon$.

We also note that \Cref{general_lower} is a markedly general result, the potential of which is by no means fully exploited in this paper. Based on the current work, which demonstrates the promise of jointly studying robustness and local differential privacy, we will continue working on understanding the interplay between these two areas in other settings, in particular for problems in high dimensions, with \Cref{general_lower} providing a minimax lower bound to start with. 

\section*{Acknowledgements}

We are thankful to the anonymous referees for detailed comments and suggestions, which greatly improved the paper.

\bibliographystyle{agsm}
\bibliography{ref}

\appendix

\section*{Appendices}

All the technical details and some additional results are provided in the Appendices.  Appendices~\ref{app-A}, \ref{app-b}, \ref{app-c} and \ref{app-d} collect all the proofs of the results in Sections~\ref{sec-introduction}, \ref{sec1}, \ref{sec2} and \ref{sec3}, respectively.  Additional results on the two-point testing problem subject to a more general class of differential privacy constraints are presented in \Cref{Sec:RDP}.  We consider the robust univariate median estimation problem under $\alpha$-LDP in \Cref{secE}, where a stochastic gradient descent type estimator is summoned to provide an optimal excess risk.

\section[]{Technical details of \Cref{sec-introduction}} \label{app-A}

\begin{proof}[Proof of \Cref{general_lower}]
Since $\varepsilon \in [0,1)$, due to the definition of $\mathcal{P}_\varepsilon(\mathcal{P})$, we have that 
\begin{equation}\label{eq-prop-1-proof-1}
    \mathcal{R}_{n,\alpha}(\theta(\mathcal{P}), \Phi \circ \rho, \varepsilon) \geq \mathcal{R}_{n,\alpha}(\theta(\mathcal{P}), \Phi \circ \rho).
\end{equation}

Using Lemma \ref{2eps} we have that when $\mathrm{TV}(R_0,R_1) \leq \varepsilon/(1-\varepsilon)$, there exist distributions $G_0$ and $G_1$ such that 
\[
    (1-\varepsilon)R_0+\varepsilon G_0 = (1-\varepsilon)R_1+\varepsilon G_1. 
\]
Writing $\widetilde{R}_i = (1-\varepsilon)R_i+\varepsilon G_i$ for $i=0,1$, it then holds that 
\begin{equation}\label{equaldist}
    \mathrm{TV}(\widetilde{R}_0,\widetilde{R}_1) = 0.
\end{equation}
We therefore have 
\begin{equation}\label{eq-prop-1-proof-2}
\mathcal{R}_{n,\alpha}(\theta(\mathcal{P}_\varepsilon), \Phi \circ \rho) \geq \frac{\Phi(\omega(\varepsilon)/2)}{2} \left(1-\sqrt{4n\alpha^2}\mathrm{TV}(\widetilde{R}_0, \widetilde{R}_1)\right) = \frac{\Phi(\omega(\varepsilon)/2)}{2},
\end{equation}
where the inequality is due to Proposition 1 in \cite{duchi2018minimax} -- a private version of Le Cam's Lemma and the identity is due to \eqref{equaldist}.

Combining \eqref{eq-prop-1-proof-1} and \eqref{eq-prop-1-proof-2}, we conclude the proof.
\end{proof}

\begin{lemma}[Theorem 5.1 in \cite{chen2018robust}]\label{2eps}
Let $R_1$ and $R_2$ be two distributions on $\mathcal{X}$. If for some $\varepsilon \in [0,1)$, we have that $\mathrm{TV}(R_1,R_2) = \varepsilon/(1-\varepsilon)$, then there exists two distributions on the same probability space $G_1$ and $G_2$ such that 
\begin{equation}\label{eq-lemma-huberthres}
    (1-\varepsilon)R_1 + \varepsilon G_1 = (1-\varepsilon)R_2+\varepsilon G_2.
\end{equation}
\end{lemma}

\begin{proof}
This is from the proof of Theorem 5.1 in \cite{chen2018robust}. Let the densities of $R_1$ and $R_2$ be
\[
r_1 = \frac{\mathrm{d}R_1}{\mathrm{d}(R_1+R_2)} \qquad \text{and} \qquad r_2 = \frac{\mathrm{d}R_2}{\mathrm{d}(R_1+R_2)}.
\]
Then the following choices of $G_1$ and $G_2$ 
\[
\frac{\mathrm{d}G_1}{\mathrm{d}(R_1+R_2)} = \frac{(r_2-r_1)\mathbbm{1}\{r_2\geq r_1\}}{\mathrm{TV}(R_1,R_2)} \qquad \text{and} \qquad \frac{\mathrm{d}G_2}{\mathrm{d}(R_1+R_2)} = \frac{(r_1-r_2)\mathbbm{1}\{r_1\geq r_2\}}{\mathrm{TV}(R_1,R_2)}
\]
satisfy \eqref{eq-lemma-huberthres}, following the calculations done in the proof of Theorem 5.1 in \cite{chen2018robust}.
\end{proof}

\section[]{Technical details of \Cref{sec1}}\label{app-b}

\subsection[]{Proofs of results in \Cref{sec1}}\label{appendix:testing}

\begin{proof}[Proof of Proposition \ref{testing_lowerbound}]
We use $QM_0^n$ and $QM_1^n$ to denote the joint distribution of $Z_1,\dotsc,Z_n$, when $X_1,\dotsc,X_n$ are generated from $M_0 \in \mathcal{P}_\varepsilon(P_0)$ and $M_1 \in \mathcal{P}_\varepsilon(P_1)$, respectively.  We then have 
\begin{align*}
    \mathcal{R}_{n,\alpha}(\varepsilon) &\geq \inf_{Q \in \mathcal{Q}_\alpha} \inf_{\phi\in \Phi_Q} \left\{QM_0^n(\phi = 1) + QM_1^n(\phi = 0)\right\} \\
    & = \inf_{Q \in \mathcal{Q}_\alpha} \inf_{\phi\in \Phi_Q} \left\{1 - [QM_0^n(\phi = 0) - QM_1^n(\phi = 0)]\right\} \geq \inf_{Q \in \mathcal{Q}_\alpha} \{1-\mathrm{TV}(QM_0^n,QM_1^n)\} \\
    & \geq \inf_{Q \in \mathcal{Q}_\alpha} \exp\left\{-\mathrm{KL}(QM_0^n,QM_1^n)\right\} \geq \exp\{-4n(e^\alpha-1)^2\mathrm{TV}^2(M_0,M_1)\},
\end{align*}
where the second inequality follows from the definition of the total variation distance, the third inequality follows from Lemma 2.6 in \cite{tsybakov2009introduction} and final inequality is due to Corollary~3 in \cite{duchi2018minimax}. 

Let $\eta=\mathrm{TV}(P_0,P_1)/\{1+\mathrm{TV}(P_0,P_1)\}$ and suppose that $\eta \geq \varepsilon$. By Lemma~\ref{2eps}, there exist probability distributions $G_0$ and $G_1$ with $(1-\eta)P_0 + \eta G_0 = (1-\eta)P_1 + \eta G_1$, so that we may write $G_0-G_1 = (1-\eta)/\eta(P_1-P_0)$. For $j=0,1$, define $M_j = (1-\varepsilon)P_j + \varepsilon G_j$. Then
\begin{align*}
    \mathrm{TV}(M_0,M_1) &= \sup_S \{ M_0(S) - M_1(S) \} = \sup_S [ (1-\varepsilon)\{P_0(S) - P_1(S)\} + \varepsilon\{G_0(S) - G_1(S)\} ] \\
    &= \sup_S [\{1-\varepsilon - \varepsilon(1-\eta)/\eta\}\{P_0(S) - P_1(S)\} ] = (1-\varepsilon/\eta) \mathrm{TV}(P_0,P_1) \\
    & = (1-\varepsilon) \{ \mathrm{TV}(P_0,P_1) - \varepsilon/(1-\varepsilon)\} \leq \mathrm{TV}(P_0,P_1) - \varepsilon/(1-\varepsilon).
\end{align*}
On the other hand, when $\eta < \varepsilon$ we take $M_j = (1-\eta)P_j + \eta G_j$. Therefore, we have $\mathcal{R}_{n,\alpha}(\varepsilon) \geq \exp{(-16n\alpha^2 \{\mathrm{TV}(P_0,P_1)-\varepsilon/(1-\varepsilon) \}_+^2 )}$ when $\alpha \in (0,1)$ since $e^\alpha-1<2\alpha$.
\end{proof}

\begin{proof}[Proof of Theorem \ref{testing_upperbound}]
The result follows from Lemma \ref{testing_lemma} since for any $P \in \mathcal{P}_\varepsilon(P_0)$, we have 
\[
\mathrm{TV}(P_0, P) = \sup_{A}|P_0(A) - P(A)| = \varepsilon \sup_{A}|P(A) - G(A)| \leq \varepsilon,
\]
and similarly for $P' \in \mathcal{P'}_\varepsilon(P_1)$. Therefore, it holds that 
\begin{align*}
    &\sup_{P \in \mathcal{P}_\varepsilon(P_0)} \mathbb{E}_{P,Q}[\tilde{\phi}] + \sup_{P' \in \mathcal{P}_\varepsilon(P_1)} \mathbb{E}_{P',Q}[(1-\tilde{\phi})] \\
    \leq & \sup_{P : \mathrm{TV}(P,P_0) \leq \varepsilon} \mathbb{E}_{P,Q}[\tilde{\phi}] + \sup_{P' :\mathrm{TV}(P',P_1)\leq \varepsilon} \mathbb{E}_{P',Q}[(1-\tilde{\phi})]\leq 4\exp\Big\{-C\alpha^2n\{\mathrm{TV}(P_0,P_1)-2\varepsilon\}^2\Big\}.
\end{align*}
\end{proof}

\begin{lemma}\label{testing_lemma}
Assume $\mathrm{TV}(P_0,P_1) >2\varepsilon$, and $\alpha \in (0,1)$, then it holds that
\begin{equation}\label{1lem2}
    \sup_{P : \mathrm{TV}(P,P_0) \leq \varepsilon} \mathbb{E}_{P,Q}[\tilde{\phi}] \leq \exp\{-C\alpha^2n\, (\mathrm{TV}(P_0,P_1)-2\varepsilon)^2\} 
\end{equation}
and
\begin{equation}\label{2lem2}
    \sup_{P' :\mathrm{TV}(P',P_1)\leq \varepsilon} \mathbb{E}_{P',Q}[(1-\tilde{\phi})] \leq \exp\{-C\alpha^2n\, (\mathrm{TV}(P_0,P_1)-2\varepsilon)^2\}, 
\end{equation}
where $C >0$ is some absolute constant.
\end{lemma}
\begin{proof}[Proof of Lemma \ref{testing_lemma}]
We only prove \eqref{1lem2} since \eqref{2lem2} follows using the same arguments. For any $P$ such that $\mathrm{TV}(P,P_0)\leq \varepsilon$, we have
\begin{align*}
    \mathbb{E}_{P,Q}[\tilde{\phi}] &= \mathbb{P}\left(|\tilde{N}_0/n - P_0(A)|>|\tilde{N}_0/n - P_1(A)|\right)\\
    & = \mathbb{P} \left( 2 \{\widetilde{N}_0/n - P(A) \} < P_0(A) + P_1(A) - 2P(A) \right) \\
    & \leq \mathbb{P} \left( 2 \{\widetilde{N}_0/n - P(A) \} < -\mathrm{TV}(P_0,P_1) + 2 \varepsilon \right),
\end{align*}
where the last line is due to $\mathrm{TV}(P_0,P_1) = P_0(A)-P_1(A)$. 

Note that 
\[
    \mathbb{E}_{P,Q}\Big[\frac{\tilde{N}_0}{n}\Big] =  \frac{e^\alpha+1}{e^\alpha-1}\left(\frac{\mathbb{E}_{P,Q}[\hat{N}_0]}{n}-\frac{1}{e^\alpha+1}\right) = P(A),
\]
since 
\[
\mathbb{E}_{P,Q}\Big[\frac{\hat{N}_0}{n}\Big] = \frac{e^\alpha}{e^\alpha+1}\mathbb{P}(Y_1 = 0)+\frac{1}{e^\alpha+1}\mathbb{P}(Y_1 = 1) = \frac{1}{e^\alpha+1}\Big\{1+(e^\alpha-1)P(A)\Big\}.
\]
Therefore, for some absolute constants $C_0$ and $C$, we have when $\mathrm{TV}(P_0,P_1) >2\varepsilon$, 
\begin{align*}
   \mathbb{E}_{P,Q}[\tilde{\phi}] 
    & \leq \mathbb{P} \left( \widetilde{N}_0/n - P(A)  < -\mathrm{TV}(P_0,P_1)/2 + \varepsilon \right) \\
 & \leq \exp\left\{-C_0\left(\frac{e^\alpha-1}{e^\alpha+1}\right)^2 n\, (\mathrm{TV}(P_0,P_1)/2-\varepsilon)^2 \right\}  \\
 & \leq \exp\{-C\alpha^2n\, (\mathrm{TV}(P_0,P_1)-2\varepsilon)^2\},
\end{align*}
where 
we apply Hoeffding's inequality \citep[e.g.\ Proposition 2.5 in][]{wainwright2019high} in the second line, since $\tilde{N}_0/n$ is a sub-Gaussian random variable with variance proxy $\sigma^2 \leq n^{-1}(e^\alpha+1)^2(e^\alpha-1)^{-2}$. 
\end{proof}

We conclude this section by giving details of the calculations presented in Remark~\ref{Rmk:TighterTesting}. Under~$\mathrm{H}_0$ we have that
\begin{align*}
    & \mathbb{E} (\widetilde{N}_0 /n) - \frac{1-\varepsilon}{2}\{P_0(A)+P_1(A)\} - \frac{\varepsilon}{2} = P(A)  - \frac{1-\varepsilon}{2}\{P_0(A)+P_1(A)\} - \frac{\varepsilon}{2} \\ 
    \geq & (1-\varepsilon)P_0(A) - \frac{1-\varepsilon}{2}\{P_0(A)+P_1(A)\} - \frac{\varepsilon}{2} = \frac{1-\varepsilon}{2} \mathrm{TV}(P_0,P_1) - \frac{\varepsilon}{2}.
\end{align*}
Similarly, under $\mathrm{H}_1$ we have
\begin{align*}
    & \mathbb{E} (\widetilde{N}_0 /n) - \frac{1-\varepsilon}{2}\{P_0(A)+P_1(A)\} - \frac{\varepsilon}{2} \\
    \leq & (1-\varepsilon)P_1(A) + \varepsilon - \frac{1-\varepsilon}{2}\{P_0(A)+P_1(A)\} - \frac{\varepsilon}{2} = \frac{\varepsilon}{2} - \frac{1-\varepsilon}{2} \mathrm{TV}(P_0,P_1).
\end{align*}
Using Hoeffding's inequality as in the proof of Lemma~\ref{testing_lemma} we see that, under $\mathrm{H}_0$,
\[
    \mathbb{P} \left( \widetilde{N}_0/n < \frac{1-\varepsilon}{2}\{P_0(A)+P_1(A)\} +  \frac{\varepsilon}{2} \right) \leq \exp\left\{ - C \alpha^2 n (1-\varepsilon)^2\{ \mathrm{TV}(P_0,P_1) - \varepsilon/(1-\varepsilon)\}_+^2 \right\}.
\]
Combining with an analogous bound under $\mathrm{H}_1$, we see that
\[
    \mathcal{R}_{n,\alpha}(\varepsilon) \leq 2\exp\left\{ - C \alpha^2 n (1-\varepsilon)^2\{ \mathrm{TV}(P_0,P_1) - \varepsilon/(1-\varepsilon)\}_+^2 \right\},
\]
and we can remove the $(1-\varepsilon)^2$ factor by noting that $\mathcal{R}_{n,\alpha}(\varepsilon)=1$ for $\varepsilon >1/2$, so we may restrict attention to $\varepsilon \leq 1/2$.

\subsection{R\'enyi Differential Privacy} \label{Sec:RDP}

In this section we consider a family of more general notions of local privacy and show that the TV distance is still the natural measure of difficulty in the simple testing problem for each constraint in the family. Thus the TV distance  arises more generally than just for LDP. There is no obvious link between the definition of the privacy constraints and the TV distance.

Following \cite{mironov2017renyi} and \cite{duchi2018right} we consider the family of local privacy constraints known as R\'enyi local differential privacy. For $\gamma \geq 2$ and distributions $P,Q$ define the $\gamma$-R\'enyi divergence
\[
    D_\gamma(P || Q) = \frac{1}{\gamma-1} \log \int \Bigl( \frac{dP}{dQ} \Bigr)^\gamma \,\mathrm{d}Q.
\]
We then say that a privacy mechanism $Q$ is $(\gamma,\alpha^2)$-R\'enyi locally differentially private (RDP) if for all $x,x' \in \mathcal{X}$ and $z_{1:(i-1)} \in \mathcal{Z}^{i-1}$ we have
\begin{equation}\label{eq-def-renyi}
     D_\gamma \bigl( Q(\cdot | x, z_{1:(i-1)}), Q(\cdot | x', z_{1:(i-1)}) \bigr) \leq \alpha^2.
\end{equation}
By the monotonicity of R\'enyi divergences \citep[e.g.~Theorem 3 in][]{van2014renyi} we have that $(\gamma,\alpha^2)$-RDP implies $(\gamma',\alpha^2)$-RDP for any $\gamma'\leq\gamma$. Hence, lower bounds over all $(2,\alpha^2)$-RDP mechanisms imply lower bounds over all $(\gamma,\alpha^2)$-RDP mechanisms for any $\gamma \geq 2$.  We would like to point out that the limit $\lim_{\gamma \to \infty}D_{\gamma}$ is the logarithm of the essential supremum of $dP/dQ$, and the limiting privacy constraint is therefore an LDP constraint \citep[e.g.~Theorem 6 in][]{van2014renyi}.  This explains that for $\gamma \in [2, \infty)$, ($\gamma, \alpha$)-RDP is a less stringent constraint than the $\alpha$-LDP constraint.  The $\alpha$ dependence can be tightened by noting that $\alpha$-LDP implies $(\gamma,2\gamma \alpha^2)$-RDP for any $\gamma \geq 1$ \citep{duchi2018right}.

For the robust testing problem we considered in \Cref{sec1}, we replace the LDP with the more general RDP and provide the following results. 

\begin{prop}
Let $\mathcal{R}_{n,\gamma,\alpha}(\varepsilon)$ be the minimax testing risk for robust testing problem defined in \eqref{eq-def-test-minimax-risk} in \Cref{sec1}, with $\mathcal{Q}_\alpha$ denoting all $(\gamma,\alpha^2)$-RDP privacy mechanisms \eqref{eq-def-renyi} with $\gamma \geq 2$.  For $\alpha \in (0, 1)$ and the robust testing problem \eqref{testing_p}, it holds that 
\[
    \mathcal{R}_{n,\gamma,\alpha}(\varepsilon) \geq \exp\bigl( - 4n \alpha^2 (1-\varepsilon)^2 \{\mathrm{TV}(P_0,P_1) - \varepsilon/(1-\varepsilon)\}_+^2 \bigr).
\]
\end{prop}
\begin{proof}
As discussed above, it suffices to study $\mathcal{R}_{n,2,\alpha}(\varepsilon)$ since we have that $\mathcal{R}_{n,\gamma,\alpha}(\varepsilon) \geq \mathcal{R}_{n,2,\alpha}(\varepsilon)$ whenever $\gamma \geq 2$. If $Q$ is any sequentially-interactive $(2,\alpha^2)$-RDP mechanism then, by Corollary~11 of \cite{duchi2018right}, we have
\[
    \mathrm{KL}( QM_0^n || QM_1^n) \leq 4 n \alpha^2 \mathrm{TV}(M_0,M_1)^2,
\]
where we use the same notation as in the proof of \Cref{testing_lowerbound}. Using the above inequality in the proof of \Cref{testing_lowerbound} we immediately have
\[
    \mathcal{R}_{n,\gamma,\alpha^2}(\varepsilon) \geq \inf_{Q \in \mathcal{Q}_{\gamma,\alpha}} \exp \{ - \mathrm{KL}( QM_0^n || QM_1^n)\} \geq \exp\{ - 4n \alpha^2 \mathrm{TV}(M_0,M_1)\}.
\]
Choosing $M_0,M_1$ as before we have the lower bound
\[
    \mathcal{R}_{n,\gamma,\alpha^2}(\varepsilon) \geq \exp \bigl( - 4n \alpha^2 (1-\varepsilon)^2 \{\mathrm{TV}(P_0,P_1) - \varepsilon/(1-\varepsilon)\}_+^2 \bigr)
\]
as claimed.
\end{proof}

In \Cref{prop-renyi_upperbound} below, we show that the exponent in the lower bound is in fact optimal up to a factor depending only on $\gamma$.
\begin{prop}\label{prop-renyi_upperbound}
Let $\mathcal{R}_{n,\gamma,\alpha}(\varepsilon)$ be the minimax testing risk for robust testing problem defined in \eqref{eq-def-test-minimax-risk} in \Cref{sec1}, with $\mathcal{Q}_\alpha$ denoting all $(\gamma,\alpha^2)$-RDP privacy mechanisms \eqref{eq-def-renyi} with $\gamma \geq 2$.  For $\alpha \in (0, 1)$ and the robust testing problem \eqref{testing_p}, it holds that
\[
    \mathcal{R}_{n,\gamma,\alpha}(\varepsilon) \leq \exp\bigl( - c_\gamma n \alpha^2 \{\mathrm{TV}(P_0,P_1) - \varepsilon/(1-\varepsilon)\}_+^2 \bigr),
\]
where $c_\gamma>0$ is a constant depending only on $\gamma$.
\end{prop}
\begin{proof}
We use a privatised Scheff\'e test as for differential privacy. Using similar notation as in~\eqref{eq-private-mechanism-test} and the randomised response mechanism with parameter $p \in (1/2,1)$, i.e.
\[
    Z_i = \begin{cases}
            Y_i, & U_i \leq p, \\
            1 - Y_i, & \mbox{otherwise}.
        \end{cases},
\]
we use the test
\[
    \phi' = \mathbbm{1} \bigl\{ 2\tilde{N}_0 /n < (1-\varepsilon)\{P_0(A) + P_1(A)\} + \varepsilon \bigr\}.
\]
With almost exactly the same proof as before, the risk of this test is bounded above by
\begin{equation}
\label{Eq:RRRisk}
    2 \exp\bigl[ -C (2p-1)^2 n \{\mathrm{TV}(P_0,P_1) - \varepsilon/(1-\varepsilon)\}_+^2 \bigr]
\end{equation}
where $C>0$ is a universal constant, provided $p \leq 3/4$, say. As previously, if $\varepsilon$ is unknown we can use a simplified test at the cost of replacing $\varepsilon/(1-\varepsilon)$ by $2 \varepsilon$ in the upper bound. Now, by Proposition~5 of \cite{mironov2017renyi}, this mechanism satisfies
\[
    \Bigl(\gamma, \frac{1}{\gamma-1} \log \bigl( p^\gamma (1-p)^{1-\gamma} + (1-p)^{\gamma} p^{1-\gamma } \bigr) \Bigr)\text{-RDP}.
\]
It remains to choose $p$ such that the mechanism satisfies $(\gamma,\alpha^2)$-RDP and plug this choice of $p$ into~\eqref{Eq:RRRisk}. Say that $p=(1+\eta)/2$ for some small $\eta>0$. Then
\begin{align*}
    &\frac{1}{\gamma-1} \log \bigl( p^\gamma (1-p)^{1-\gamma} + (1-p)^{\gamma} p^{1-\gamma } \bigr) \\
    &= \frac{1}{\gamma -1 } \log \biggl( \frac{1}{2} \biggl\{ (1-\eta) \Bigl( \frac{1+\eta}{1-\eta} \Bigr)^\gamma + (1+\eta)  \Bigl( \frac{1+\eta}{1-\eta} \Bigr)^{-\gamma} \biggr\} \biggr) \\
    & \leq \frac{1}{2(\gamma-1)} \biggl[ (1-\eta) \biggl\{  \Bigl( \frac{1+\eta}{1-\eta} \Bigr)^\gamma -1 \biggr\} + (1+\eta) \biggl\{  \Bigl( \frac{1+\eta}{1-\eta} \Bigr)^{-\gamma} -1 \biggr\} \biggr] \\
    & = \frac{1}{2(\gamma-1)} \biggl[ (1-\eta) \biggl\{  \Bigl( \frac{1+\eta}{1-\eta} \Bigr)^\gamma -1 - 2 \gamma \eta \biggr\} + (1+\eta) \biggl\{  \Bigl( \frac{1+\eta}{1-\eta} \Bigr)^{-\gamma} -1 + 2 \gamma \eta  \biggr\} - 4 \gamma \eta^2 \biggr] \\
    & \sim 2 \gamma \eta^2
\end{align*}
as $\eta \searrow 0$. It is now clear that there exists $C_\gamma$ such that
\[
    \frac{1}{\gamma-1} \log \bigl( p^\gamma (1-p)^{1-\gamma} + (1-p)^{\gamma} p^{1-\gamma } \bigr) \leq C_\gamma^2 \eta^2
\]
for all $\eta \in [0,1/2]$. Hence, if $p=1/2+\alpha/(2C_\gamma)$ the mechanism satisfies $(\gamma,\alpha^2)$-RDP. Thus
\[
    \mathcal{R}_{n,\gamma,\alpha}(\varepsilon) \leq 2 \exp \bigl[ - (C/C_\gamma^2) n \alpha^2 \{\mathrm{TV}(P_0,P_1) - \varepsilon/(1-\varepsilon)\}_+^2 \bigr]
\]
as claimed.
\end{proof}

\section[]{Technical details of \Cref{sec2}}\label{app-c}

\subsection[]{Proofs of results in \Cref{sec2}}

\begin{proof}[Proof of \Cref{mean_lowerbound}]

In light of \Cref{general_lower}, it suffices to consider the minimax risk without contamination, denoted as $\mathcal{R}_{n, \alpha}(0)$, and $\omega^2(\varepsilon)$, defined in \eqref{eq-TV-modulus-conti}, respectively. 

\medskip
\noindent \textbf{Step 1.}  In this step, we are to lower bound $\mathcal{R}_{n, \alpha}(0)$ in two aspects, with and without the effect of $D$, respectively.

Recalling the total variation modulus of continuity $\omega(\varepsilon)$, for $\eta > 0$, let
\begin{equation}\label{eq-omega-prime}
    \omega'(\eta) = \omega\left(\frac{\eta}{1+\eta}\right) = \sup\{|\theta(R_0') - \theta(R_1')|: \, \mathrm{TV}(R_0', R_1') \leq \eta,\, R_0', \, R_1' \in \mathcal{P}_k\},
\end{equation}
where $\theta(\cdot)$ denotes the expectation.

Consider the following distributions $R_0$ and $R_1$.  Let 
\[
    R_0(\{D-1\}) = 1 \quad \mbox{and} \quad \begin{cases}
        R_1(\{D-1\}) = 1-\eta, \\
        R_1\left(\{D-1+\left(2\eta\right)^{-1/k}\}\right) = \eta, 
    \end{cases} \, \eta \in (0, 1),
\]
which satisfy that $R_0, R_1 \in \mathcal{P}_k$, defined in \eqref{eq-P-k-dist-def}.  To be specific, as for $R_0$, it holds that $\mu_0 = \mathbb{E}_{R_0}(X) = D-1 \in [-D,D]$ and $\sigma_0^k = \mathbb{E}_{R_0}[|X-\mu_0|^k] = 0 < 1$; as for $R_1$, it holds that 
\[
    \mu_1 = \mathbb{E}_{R_1}(X) = (1-\eta)(D-1)+\eta \left(\{D-1+\left(2\eta\right)^{-1/k}\}\right) = D-1+ 2^{-1/k}\eta^{1-1/k} \leq D
\]
and 
\begin{align*}
    \sigma_1^k &= \mathbb{E}_{R_1}[|X-\mu_1|^k] = (1-\eta) \Big(2^{-1/k}\eta^{1-1/k}\Big)^k + \eta \left\{(2\eta)^{-1/k} - (2^{-1/k}\eta^{1-1/k}\right\}^k\\
    & = 2^{-1}(1-\eta)\eta^{k-1} + 2^{-1}(1-\eta)^k \leq 1,
\end{align*}
since $\eta<1<2^{1/k}$.
  
Notice that $\mathrm{TV}(R_0,R_1) = \eta$ by construction and $|\mu_1-\mu_0| = 2^{-1/k}\eta^{1-1/k}$.  We therefore have, due to the definition \eqref{eq-omega-prime}, that 
\begin{equation}\label{eq-omega-prime-lower-bound}
    \omega'(\eta) \geq 2^{-1/k}\eta^{1-1/k}.    
\end{equation}
Due to Corollary 3.1 in \cite{rohde2020geometrizing}, we have that for $n$ larger than some absolute constant $n_0$, 
\begin{equation}\label{eq-prop-4-pf-1}
    \mathcal{R}_{n,\alpha}(0) \geq C_1\left\{\omega'(C_0(n\alpha^2)^{-1/2})\right\}^2 \gtrsim (n\alpha^2)^{1/k-1}, 
\end{equation}
where $C_0, C_1 > 0$ are absolute constants.  

It follows from \Cref{tomnotes}, we have that 
\begin{equation}\label{eq-prop-4-pf-2}
    \mathcal{R}_{n,\alpha}(0) \geq \frac{D^2}{32\exp{(64n\alpha^2)}}.
\end{equation}

Combining \eqref{eq-prop-4-pf-1} and \eqref{eq-prop-4-pf-2}, we have that 
\begin{equation}\label{eq-prop-4-pf-3}
    \mathcal{R}_{n,\alpha}(0) \gtrsim (n\alpha^2)^{1/k-1} \vee \frac{D^2}{\exp{(64n\alpha^2)}}.
\end{equation}

\medskip
\noindent \textbf{Step 2.}  To apply \Cref{general_lower}, it suffices to lower bound $\omega^2(\varepsilon/2)$.  We have that
\begin{equation}\label{eq-prop-4-pf-4}
    \omega^2(\varepsilon/2) = \omega'\left(\frac{\varepsilon}{2 - \varepsilon}\right) \gtrsim \varepsilon^{2 - 2/k},
\end{equation}
where the inequality is due to \eqref{eq-omega-prime-lower-bound}.

Combining \eqref{eq-prop-4-pf-3} and \eqref{eq-prop-4-pf-4}, in light of \Cref{general_lower}, we complete the proof.
\end{proof}

\begin{lemma}\label{tomnotes}
Under the same settings as those in \Cref{mean_lowerbound} and assuming $\varepsilon = 0$, it holds that 
\[
\mathcal{R}_{n,\alpha}(0) \geq \frac{D^2}{32\exp{(64n\alpha^2)}} 
\]
\end{lemma}
\begin{proof}
 Let $M(\delta)$ denote the $\delta$-packing number of the interval $[-D,D]$, i.e.\ the maximal cardinality of a set $\{\mu_1,\dotsc,\mu_M\} \subseteq [-D,D]$ such that $|\mu_i-\mu_j| > \delta$ for all distinct $i,j \in \{1,\dotsc,M\}$. Let $J$ be uniformly distributed on $\{1,\dotsc,M(\delta)\}$ and $P_i \in \mathcal{P}_k$ be distribution with mean $\mu_i$, for $i=1,\dotsc,M(\delta)$. We use $QP^n_i$ to denote the marginal distribution of the private data $Z_1,\dotsc,Z_n$, and define the mixture distribution $\overline{Q} = \{M(\delta)\}^{-1}\sum_{i=1}^{M(\delta)}QP^n_i$. Then using (15) in \cite{duchi2018minimax}, we can upper bound the mutual information $I(Z_1^n; J)$ as
\begin{equation}\label{fano_upper}
    I(Z_1^n; J) = \frac{1}{M(\delta)}  \sum_{i=1}^{M(\delta)} \mathrm{KL}(QP^n_i, \overline{Q}) \leq 4n(e^\alpha-1)^2 \frac{1}{\{M(\delta)\}^2} \sum_{i,i'=1}^{M(\delta)} \mathrm{TV}(P_i,P_{i'}) \leq 4n(e^\alpha-1)^2,
\end{equation}
where the last inequality holds from upper bounding the total variation distance by $1$. Then, applying Fano's inequality \citep[e.g.\ Proposition 15.12 in][]{wainwright2019high} together with \eqref{fano_upper}, we have that
\[
\mathcal{R}_{n,\alpha}(0) \geq \frac{\delta^2}{4} \inf_{Q \in \mathcal{Q}_\alpha} \Big\{1-\frac{I(Z_{1}^n;J)+\log(2)}{\log(M(\delta))}\Big\} \geq \frac{\delta^2}{4} \Big\{1-\frac{4n(e^\alpha-1)^2+\log(2)}{\log(M(\delta))}\Big\}.
\]
Writing $M^{-1}(x)=\inf\{\delta \in (0,\delta_0): M(\delta)<x\}$, for some $\delta_0>0$ and $x>1$, and note that when $D<\infty$, $M^{-1}(x) = 2D/(\lceil x \rceil -1) \geq 2D/x$. Therefore, it holds that when $\alpha \in (0,1)$
\[
\mathcal{R}_{n,\alpha}(0) \geq \frac{1}{8} \left(M^{-1}\left(4e^{8n(e^\alpha-1)^2}\right)\right)^2 \geq \frac{D^2}{32\exp{(16n(e^\alpha-1)^2)}} \geq \frac{D^2}{32\exp{(64n\alpha^2)}},
\]
where the last inequality is due to $e^\alpha-1<2\alpha$ when $\alpha \in (0,1)$.
\end{proof}

\begin{proof}[Proof of \Cref{mean_optimalthm}]

Fix the data-generating distribution $P_{k,\varepsilon}$, write $P_0$ for the uncontaminated version, and let
\[
    \mathcal{J}_0 = \{j \in \mathcal{J} : P_{k,\varepsilon}(A_j) > \varepsilon + (1-\varepsilon)(6/M)^k \}.
\]
Since, by Chebyshev's inequality,
\[
    \mathbb{P}(|X-\mu| \geq M/6) \leq \varepsilon + (1-\varepsilon)(6/M)^k,
\]
we must always have $|\mathcal{J}_0| \leq 2$. Again by Chebyshev's inequality there will always exist $j \in \mathcal{J}$ with $P_{k,\varepsilon}(A_j) \geq (1/2)\{1-\varepsilon-(1-\varepsilon)(6/M)^k\}$ so, since we have the condition $\varepsilon+(1-\varepsilon)(6/M)^k\leq 1/12 < 1/3$, we must also have $|\mathcal{J}_0| \geq 1$, so that $|\mathcal{J}_0| \in \{1,2\}$.  Now, for $j \not\in \mathcal{J}_0$ we can see that
\begin{align*}
    \mathbb{P}(j \in \hat{\mathcal{J}}) &= \mathbb{P} \biggl( \frac{1}{n} \sum_{i=1}^n \Bigl( \mathbbm{1}_{\{X_i \in A_j\}} + \frac{2}{\alpha} W_{ij} \Bigr) \geq \tau \biggr) \\
    & \leq \mathbb{P} \biggl( \frac{1}{n} \sum_{i=1}^n \Bigl( \mathbbm{1}_{\{X_i \in A_j\}} - P_{k,\varepsilon}(A_j) + \frac{2}{\alpha} W_{ij} \Bigr) \geq 4\sqrt{2 \log(12T/(M\delta))/(n\alpha^2)} \biggr) \\
    & \leq \exp \biggl( - \frac{4 \log(12T/(M\delta))}{\alpha^2} \biggr) + \exp \biggl( - \frac{1}{8} \min \biggl\{ 8 \log(12T/(M\delta)), 4 \sqrt{2n \log(12T/(M\delta))} \biggr\} \biggr) \\
    & \leq 2 \exp \bigl( - \log(12T/(M\delta)) \bigr) = M \delta /(6T),
\end{align*}
where the second inequality follow from Hoeffding's inequality \citep[e.g.\ Theorem 2.26 in][]{vershynin2018high} and Bernstein's inequality \citep[e.g.\ Theorem 2.8.1 in][]{vershynin2018high}, and for the final inequality we used the condition $n^{-1} \log(12T^3 n \alpha^2/M) \leq 1/2$. Since $|\mathcal{J} \setminus \mathcal{J}_0| \leq 6T/M$ we therefore have that
\begin{equation}
\label{Eq:JHatInclusion}
    \mathbb{P}(\hat{\mathcal{J}} \not \subseteq \mathcal{J}_0) \leq \delta.
\end{equation}
We now show that $\hat{\mathcal{J}}$ is non-empty with high probability. Indeed, let $j\in\mathcal{J}$ be such that $P_{k,\varepsilon}(A_j) \geq (1/2)\{1-\varepsilon-(1-\varepsilon)(6/M)^k\}$. Then, using the same concentration inequality as above, we have
\begin{align}
\label{Eq:JHatEmpty}
    \mathbb{P}( \hat{\mathcal{J}} = \emptyset ) &\leq \mathbb{P}(j \not\in \hat{\mathcal{J}}) = \mathbb{P} \biggl( \frac{1}{n} \sum_{i=1}^n\bigl\{ \mathbbm{1}_{\{X_i \in A_j\}} + (2/\alpha)W_{ij} \bigr\} < \tau \biggr) \nonumber \\ 
    & \leq \mathbb{P} \biggl( \frac{1}{n} \sum_{i=1}^n\bigl\{ \mathbbm{1}_{\{X_i \in A_j\}} - P_{k,\varepsilon}(A_j) + (2/\alpha)W_{ij} \bigr\} < \tau - (1/2)\{ 1 - \varepsilon - (1-\varepsilon)(6/M)^k\} \biggr) \nonumber \\
    & \leq \mathbb{P} \biggl( \frac{1}{n} \sum_{i=1}^n\bigl\{ \mathbbm{1}_{\{X_i \in A_j\}} - P_{k,\varepsilon}(A_j) + (2/\alpha)W_{ij} \bigr\} < -1/4 \biggr) \nonumber \\
    & \leq \exp\Bigl( - \frac{n}{128} \Bigr) + \exp \Bigl( - \frac{1}{8} \min(n\alpha^2/64, n \alpha/4) \Bigr) \leq 2 \exp \Bigl( - \frac{n \alpha^2}{512} \Bigr),
\end{align}
by using the conditions $\alpha \leq 1$, $\varepsilon+(1-\varepsilon)(6/M)^k \leq 1/12$ and $(n\alpha^2)^{-1} \log(12T^3 n \alpha^2/M) \leq 
1/512$. 

With these bounds in place we turn to the mean squared error of our estimator, which we decompose as
\begin{equation}
\label{Eq:MSEDecomposition}
    \mathbb{E}\{ (\hat{\mu} - \mu)^2 \} = (\mathbb{E} \hat{\mu} - \mu)^2 + \mathrm{Var}\bigl( \mathbb{E} \bigl\{ \hat{\mu} | J \bigr\} \bigr) + \mathbb{E} \bigl\{ \mathrm{Var}( \hat{\mu} |J ) \bigr\}.
\end{equation}
The final term of~\eqref{Eq:MSEDecomposition} can be simply bounded by saying
\begin{equation}
\label{Eq:ConditionalVariance}
    \mathbb{E} \bigl\{ \mathrm{Var}( \hat{\mu} |J ) \bigr\} =  \frac{1}{n} \mathbb{E} \bigl\{ \mathbbm{1}_{\{ \hat{\mathcal{J}} \neq \emptyset\}} \mathrm{Var}(Z_{(L+2)n}^{(L)} | L) \bigr\} \leq \frac{1}{n} \biggl( \frac{M^2}{4} + \frac{2M^2}{\alpha^2} \biggr) = \frac{9M^2}{4n \alpha^2}.
\end{equation}
Bounds on the first two terms of~\eqref{Eq:MSEDecomposition} will follow from~\eqref{Eq:JHatInclusion},~\eqref{Eq:JHatEmpty} and control of $\mathbb{E}(\hat{\mu} | J=j_0-1)$ for $j_0 \in \mathcal{J}_0$. Given a realisation $j_0 \in \mathcal{J}_0$ of $J+1$, write $\ell_0$ for the associated realisation of $L$, so that $\ell_0 \equiv (j_0-1) (\text{mod} 3)$. Then we have
\begin{align}
\label{Eq:ConditionalMeanDecomposition}
    \bigl| \mathbb{E}(\hat{\mu} | J=j_0-1) - \mu \bigr| &= \bigl| \mathbb{E}( [ R_{(\ell_0+2)n}^{(\ell_0)} ]_0^M ) + (j_0-2)M/3 - \mu \bigr| \nonumber \\
    & \leq \varepsilon \max\bigl( | (j_0+1)M/3 - \mu|, |(j_0-2)M/3 - \mu| \bigr) \nonumber \\
    & \hspace{75pt} + (1-\varepsilon) \bigl| \mathbb{E}_{P_0} \bigl\{ [ R_{(\ell_0+2)n}^{(\ell_0)} ]_0^M + (j_0-2)M/3 - X_{(\ell_0+2)n} \bigr\} \bigr|.
\end{align}
Since $j_0 \in \mathcal{J}_0$, if it were the case that $\mu \not\in [(j_0-3/2)M/3, (j_0+1/2)M/3)$ we would have the chain of inequalities
\begin{align*}
    \varepsilon + (1-\varepsilon)(6/M)^k < P_{k,\varepsilon}(A_{j_0}) \leq \varepsilon + (1-\varepsilon) \mathbb{P}_{P_0}(|X-\mu| \geq M/6) \leq \varepsilon + (1-\varepsilon)(6/M)^k,
\end{align*}
which is a contradiction. We therefore have $\mu \in [(j_0-3/2)M/3, (j_0+1/2)M/3)$ and so
\begin{equation}
\label{Eq:ContaminationBias}
    \varepsilon \max\bigl( | (j_0+1)M/3 - \mu|, |(j_0-2)M/3 - \mu| \bigr) \leq 5 \varepsilon M/2.
\end{equation}
Now, again using the fact that $\mu \in [(j_0-3/2)M/3, (j_0+1/2)M/3)$, we see that
\begin{align}
\label{Eq:TruncationBias}
    \bigl| \mathbb{E}_{P_0} & \bigl\{ [ R_{(\ell_0+2)n}^{(\ell_0)} ]_0^M + (j_0-2)M/3 - X_{(\ell_0+2)n} \bigr\} \bigr| \nonumber \\
    & = \Bigl| \mathbb{E}_{P_0} \Bigl\{ \mathbbm{1}_{\{ X_{(\ell_0+2)n} \not\in [(j_0-2)M/3, (j_0+1)M/3)\} } \Bigl( [R_{(\ell_0+2)n}^{(\ell_0)} ]_0^M + (j_0-2)M/3 - X_{(\ell_0+2)n} \Bigr) \Bigr\} \Bigr| \nonumber \\
    & \leq \mathbb{E}_{P_0} \Bigl\{ \mathbbm{1}_{\{ X_{(\ell_0+2)n} < (j_0-2)M/3\} } \Bigl( [R_{(\ell_0+2)n}^{(\ell_0)} ]_0^M + (j_0-2)M/3 - X_{(\ell_0+2)n} \Bigr) \Bigr\} \nonumber \\
    & \hspace{50pt} + \mathbb{E}_{P_0} \Bigl\{ \mathbbm{1}_{\{ X_{(\ell_0+2)n} \geq  (j_0+1)M/3\} } \Bigl( X_{(\ell_0+2)n} - [R_{(\ell_0+2)n}^{(\ell_0)} ]_0^M - (j_0-2)M/3  \Bigr) \Bigr\} \nonumber \\
    & \leq \mathbb{E}_{P_0} \Bigl\{ \mathbbm{1}_{\{ X_1 - \mu < - M/6 \} } \Bigl( M - \frac{M}{6} + \mu - X_1 \Bigr) \Bigr\} \nonumber \\ 
    & \hspace{50pt} + \mathbb{E}_{P_0} \Bigl\{ \mathbbm{1}_{\{ X_1 - \mu \geq  M/6 \} } \Bigl( X_1 - \mu + (j_0+1/2)M/3 - (j_0-2)M/3 \Bigr) \Bigr\} \nonumber \\
    & = \frac{5M}{6} \mathbb{P}_{P_0}( |X_1 - \mu| \geq M/6) + \mathbb{E}_{P_0}( \mathbbm{1}_{\{|X_1 - \mu | \geq M/6\}} |X_1-\mu| ) \nonumber \\
    & \leq \frac{5M}{6} \Bigl( \frac{6}{M} \Bigr)^k + \Bigl(\frac{6}{M} \Bigr)^{k-1} = \frac{6^k}{M^{k-1}}.
\end{align}
Combining~\eqref{Eq:ConditionalMeanDecomposition},~\eqref{Eq:ContaminationBias} and~\eqref{Eq:TruncationBias} we see that, when $j_0 \in \mathcal{J}_0$, we have
\[
    \bigl| \mathbb{E}( \hat{\mu} | J=j_0-1) - \mu \bigr| \leq \frac{5}{2} \varepsilon M + (1-\varepsilon) \frac{6^k}{M^{k-1}}.
\]
First, we use this and the fact that $-T-M \leq \mathbb{E}(\hat{\mu} | J) \leq T+ M$ to see that
\begin{align}
\label{Eq:Bias}
    | \mathbb{E} \hat{\mu} - \mu| &\leq T \mathbb{P}(\hat{\mathcal{J}} = \emptyset) +  (2T+M) \mathbb{P}( \hat{\mathcal{J}} \not \subseteq \mathcal{J}_0) + \biggl| \sum_{j_0 \in \mathcal{J}_0 } \mathbb{P}(J=j_0-1) \bigl\{ \mathbb{E}( \hat{\mu} | J=j_0-1) - \mu \bigr\} \biggr| \nonumber \\
    & \leq 2T\exp(-n \alpha^2/512) +  \delta(2T+M) + \frac{5}{2} \varepsilon M + (1-\varepsilon) \frac{6^k}{M^{k-1}},
\end{align}
which controls the bias of our estimator. In dealing with the variance of the conditional mean, it will be helpful to split into the two cases $|\mathcal{J}_0|=1$ and $|\mathcal{J}_0|=2$. In the first of these cases, writing $\mathcal{J}_0=\{j_0\}$, we can say by~\eqref{Eq:JHatInclusion} and~\eqref{Eq:JHatEmpty} that
\begin{equation}
\label{Eq:ConditionalMean1}
    \mathrm{Var}\bigl( \mathbb{E} \bigl\{ \hat{\mu} | J \bigr\} \bigr) \leq \mathbb{E} \Bigl[ \Bigl\{ \mathbb{E}( \hat{\mu} | J ) - \mathbb{E}( \hat{\mu} | J = j_0-1) \Bigr\}^2 \Bigr] \leq (T+M)^2 \delta + 8(T+M)^2\exp(-n \alpha^2/512).
\end{equation}
In the second case, writing $\mathcal{J}_0=\{j_0,j_0+1\}$, we can use~\eqref{Eq:JHatInclusion} and~\eqref{Eq:JHatEmpty} to say that
\begin{align}
\label{Eq:ConditionalMean2}
     \mathrm{Var}\bigl(& \mathbb{E} \bigl( \hat{\mu} | J \bigr) \bigr) \leq \mathbb{E} \biggl[ \biggl\{ \mathbb{E}( \hat{\mu} | J ) - \sum_{j \in \mathcal{J}_0} \mathbb{P}(J=j-1)\mathbb{E}( \hat{\mu} | J = j-1) \biggr\}^2 \biggr] \nonumber \\
     & \leq (T+M)^2 \delta + 8(T+M)^2\exp(-n \alpha^2/512)  +  2\bigl\{ \mathbb{E}( \hat{\mu} | J=j_0-1) - \mathbb{E}( \hat{\mu} | J=j_0) \bigr\}^2 \nonumber \\
     & \leq (T+M)^2 \delta + 8(T+M)^2\exp(-n \alpha^2/512) + 2\{ 5 \varepsilon M + 2(1-\varepsilon) 6^k M^{-(k-1)} \}^2 \nonumber \\
     & \leq (T+M)^2 \delta + 8(T+M)^2\exp(-n \alpha^2/512) + 100 \varepsilon^2 M^2 + 16(1-\varepsilon)^2 36^k M^{-2(k-1)}.
\end{align}
Combining~\eqref{Eq:MSEDecomposition},~\eqref{Eq:ConditionalVariance},~\eqref{Eq:Bias},~\eqref{Eq:ConditionalMean1} and~\eqref{Eq:ConditionalMean2}, we see that we have
\[
    \mathbb{E}\{ (\hat{\mu}- \mu)^2 \}  \lesssim T^2\exp(-n \alpha^2/512) +  \frac{M^2}{n \alpha^2} + \varepsilon^2 M^2 + M^{-2(k-1)},
\]
as claimed.

\end{proof}

\subsection{Additional results}\label{sec-duchi-mean}
\begin{lemma}\label{lemma-mean-alpha-ldp}
For $i = 1,\dotsc, n$, write $Z_i = (Z_{i1},\dotsc,Z_{iJ})$, where $Z_{ij}$'s are defined in \eqref{eq-priv-fold1}. Then $Z_{i}$ satisfy $\alpha$-LDP. For $i = n+1,\dotsc,4n$, $Z_i^{(\ell)}$ defined in \eqref{eq-priv-foldl} also satisfy $\alpha$-LDP. 
\end{lemma}
\begin{proof}
For $i = n+1,\dotsc,4n$, $Z_i^{(\ell)}$ are obtained directly using the Laplace mechanism \citep{dwork2006calibrating} and  satisfy $\alpha$-LDP since $[R_i^{(\ell)}]_0^M \in [0,M]$ for any value $x_i$. For $i = 1,\dotsc,n$, we have 
\begin{align*}
    \frac{\mathbb{P}(Z_i = z_i|x_i)}{\mathbb{P}(Z_i = z_i|x_i')} &= \prod_{j \in \mathcal{J}} \exp\left(\frac{|z_{ij} - \mathbbm{1}_{x_i \in A_j}|-|z_{ij}-\mathbbm{1}_{x_j'\in A_j}|}{2/\alpha}\right)\\& \leq \prod_{j \in \mathcal{J}} \exp\left(\frac{|\mathbbm{1}_{\{x_i \in A_j\}}-\mathbbm{1}_{\{x_j'\in A_j\}}|}{2/\alpha}\right) \leq \exp(\alpha),
\end{align*}
where the last inequality holds since there are two most two $j$'s in $\mathcal{J}$ such that $\mathbbm{1}_{\{x_i \in A_j\}} \neq \mathbbm{1}_{\{x_j'\in A_j\}}$. 
\end{proof}

We conclude this section by an analysis of the robustness property of the following mean estimation procedure proposed by \cite{duchi2018minimax}. For a truncation level $M > 0$ to be chosen, the privatised data $Z_i$'s are obtained by 
\[
Z_i = [X_i]_M + \frac{2M}{\alpha}W_i,
\]
where $W_i$, $i = 1\dotsc,n$ are i.i.d.~standard Laplace random variables, and the mean is estimated by $$\hat{\mu} = n^{-1}\sum_{i=1}^nZ_i.$$

\begin{prop}\label{meanupper1}
Given i.i.d random variables $X_1,\dotsc,X_n \sim P_{k, \varepsilon} = (1-\varepsilon)P_k+\varepsilon G$, where $P_k \in \mathcal{P}_k$. Suppose that $M>2D$, then for $\alpha \in (0,1]$, it holds that
\begin{equation}\label{mean_upper_1}
    \mathbb{E}[|\hat{\mu} - \mu|^2] \lesssim \frac{M^2}{n\alpha^2}+\varepsilon^2M^2 + \frac{D^2  }{M^{2k-2}}.
\end{equation}
With the choice of truncation parameter $M = D^{1/k}\left(\varepsilon^{-1/k} \wedge (n\alpha^2)^{1/(2k)}\right)$, we have
\[
\mathbb{E}[|\hat{\mu} - \mu|^2] \lesssim D^{2/k}\left((n\alpha^2)^{1/k-1} \vee \varepsilon^{2-2/k}\right).
\]

\end{prop}

\begin{remark}
From Proposition \ref{meanupper1}, we have when $D$ is of constant order, $\hat{\mu}$ matches the minimax lower bound, but when $D$ increases with $n$, this bound is sub-optimal. 
\end{remark}
\begin{proof} 
We use the bias and variance decomposition $\mathbb{E}[|\hat{\mu}-\mu|^2] = (\mathbb{E}[\hat{\mu}-\mu])^2+\mathrm{Var}(\hat{\mu})$ to bound the mean squared error. 
Due to the truncation, the variance of $\hat{\mu}$ can be bounded by 
\[
\mathrm{Var}(\hat{\mu})\leq \frac{1}{n}\mathrm{Var}([X_i]_M]) + \frac{8M^2}{n\alpha^2} \leq \frac{9M^2}{n\alpha^2}.
\]
As for the bias term, since $W_i$'s have mean zero, it holds that  
\[  
\left|\mathbb{E}_{}[\hat{\mu} - \mu]\right| = \left|\mathbb{E}_{}[[X_1]_M - \mu]\right| \leq \varepsilon (M+|\mu|) + (1-\varepsilon)|\mathbb{E}_{P_k}[[X]_M]-\mu| \leq \varepsilon|M+D|+\left|\mathbb{E}_{P_k}[[X]_M]-\mu\right|.
\]
The final term is controlled in a way that is similar to equation (45) in \cite{duchi2018minimax}. Provided that $M>2D$, we have 
\begin{align*}\label{truncation}
     & \left|\mathbb{E}_{P_k}[[X]_M]-\mu\right| \leq \mathbb{E}_{P_k}\big|[[X]_M-X]\big| \leq 2^{1/k} \mathbb{E}_{P_k}[|X|^k]^{1/k}  \mathbb{P}(|X|>M)^{1/k'} \\
     \leq & 2^{1+1/k}(1+|\mu|)(M-\mu)^{1-k} \leq \frac{8D}{(M-D)^{k-1}} \lesssim \frac{D}{M^{k-1}}.
\end{align*}
where $k' = k/(k-1)$, and in the third inequality we used 
\[
\mathbb{E}_{P_k}[|X|^k]^{1/k} = \mathbb{E}_{P_k}[|X-\mu+\mu|^k]^{1/k} \leq 2(\mathbb{E}_{P_k}|X-\mu|^k+|\mu|^k)^{1/k} \leq 2(1+D)
\]
and 
\[
\mathbb{P}(|X|>M)^{1/k'} = \mathbb{P}(|X -\mu|>M-\mu)^{1/k'} \leq  (M-\mu)^{1-k} 
\]
by Markov's inequality. 
 Therefore, we have finally 
 \[
 \mathbb{E}_{}[|\hat{\mu} - \mu|^2] \lesssim \frac{M^2}{n\alpha^2}+\varepsilon^2M^2 + \frac{D^2  }{M^{2k-2}}.
 \]
\end{proof}

\section[]{Technical details of \Cref{sec3}}\label{app-d}

\subsection[]{Proofs of results in \Cref{sec3}}

\begin{proof}[Proof of Proposition \ref{npd_lowerbound_prop}]
In light of \Cref{general_lower}, it suffices to consider the minimax risks without contamination, denoted as $\mathcal{R}_{n, \alpha}(\mathcal{F}_{\beta}, \|\cdot\|^2_2, 0)$ and $\mathcal{R}_{n, \alpha}(\mathcal{F}_{\beta}, \|\cdot\|_{\infty}, 0)$; and $\omega^2(\varepsilon)$, defined in~\eqref{eq-TV-modulus-conti}.

\medskip
\noindent \textbf{Step 1.}  When $\varepsilon = 0$, claims (\ref{npd_lowerbound_2}) and (\ref{npd_lowerbound_inf}) follow from Corollary 2.1 (or equation (1.6)) in \cite{butucea2020local} since Sobolev space is a special case of Besov space studied therein. Specifically, with $s$ replaced by $\beta$ and $p = 2$ in their result, we have 
\begin{equation}\label{eq-density-nocta}
    \mathcal{R}_{n, \alpha}(\mathcal{F}_{\beta}, \|\cdot\|^2_2, 0) \gtrsim (n\alpha^2)^{-\frac{2\beta}{2\beta+2}} \qquad \text{and} \qquad \mathcal{R}_{n, \alpha}(\mathcal{F}_{\beta}, \|\cdot\|_{\infty}, 0) \gtrsim \left\{\frac{\log(n\alpha^2)}{n\alpha^2}\right\}^{\frac{2\beta-1}{4\beta+2}}
\end{equation}

Note that although the statement of Corollary 2.1 in \cite{butucea2020local} does not include the case $r = \infty$, an inspection of the proof suggests that the $r = \infty$ case follows exactly the same arguments and corresponding result holds by first taking $r$-th root and then choose $r = \infty$.  

\medskip 
\noindent \textbf{Step 2.}  To obtain a lower bound on $\omega(\varepsilon)$, due to its definition, it suffices to have a lower bound on different losses of a particular pair of densities, satisfying the total variation distance upper bound.  To be specific, we consider the densities $f_0 \in \mathcal{F}_\beta$ and $f_1 = f_0+\gamma \psi_{jk}$, where 
the $f_0 \in \mathcal{F}_\beta$ satisfies $f_0 = c_0 > 0$ on some interval $[a,b] \subseteq [0,1]$ and the function $\psi_{jk}(x) = 2^{j/2}\psi(2^jx-k)$ is a wavelet basis function that is support on $[a,b]$. Note that $f_1$ is a indeed a density function: (1) $\int_{0}^1 \psi_{jk}(x)\, \mathrm{d}x = 0$, since \eqref{regular}; and (2) $f_1 \geq 0$, since with $\beta>1/2$ and $c>0$ sufficiently small, it holds that 
\[
f_1 \geq c_0-\gamma\|\psi_{jk}\|_\infty = c_0 - c2^{-j(\beta-1/2)}\|\psi\|_{\infty} \geq c_0 - c\|\psi\|_{\infty} \geq 0.
\]
In addition, due to the characterisation \eqref{besov_sobolev} of Sobolev space and the fact that $2^{j\beta} \gamma = c < \infty$, we have that $f_1 \in \mathcal{F}_\beta$.

Given that $f_0, f_1 \in \mathcal{F}_\beta$, we also have that 
\begin{equation}\label{eq-pf-prop6-1}
\mathrm{TV}(P_{f_0},P_{f_1}) = \frac{1}{2}\int_{0}^1 |f_0(x) - f_1(x)|\,\mathrm{d}x = \frac{\gamma}{2}\int_{0}^1 |\psi_{jk}(x)|\,\mathrm{d}x = \frac{c}{2}\|\psi\|_12^{-j(\beta+1/2)}. 
\end{equation}
Note that $\|\psi\|_t = \left(\int_{-A}^{A}|\psi(x)|^t\,\mathrm{d}x\right)^{1/t}< \infty$ for any integer $t \geq 1$ due to our assumption $\|\psi\|_\infty < \infty$ in \eqref{regular} and $\psi$ is compactly supported on $[-A,A]$. With the choice $2^j = c_1\varepsilon^{-1/(\beta+1/2)}$, for some absolute constant $c_1$ large enough, \eqref{eq-pf-prop6-1} leads to that $\mathrm{TV}(P_{f_0},P_{f_1}) \leq \varepsilon/(1-\varepsilon)$.  Then for the $t$-th norm, with $t$ to be specified, we have that
\begin{equation}\label{eq-pf-prop6-2}
\|f_1-f_0\|_t = \gamma2^{j(1/2-1/t)}\|\psi\|_t \asymp \varepsilon^{\beta/(\beta+1/2)} \varepsilon^{(1/t-1/2)/(\beta+1/2)} = \varepsilon^{(\beta+1/t-1/2)/(\beta+1/2)},
\end{equation}
which serves as a lower bound on $\omega(\varepsilon)$ with $\rho = \|\cdot\|_t$.  

Setting $t = 2$, \eqref{eq-pf-prop6-2} leads to
\begin{equation}\label{eq-pf-prop6-3}
    \omega(\varepsilon/2) \gtrsim \begin{cases}
        \varepsilon^{\frac{4\beta}{2\beta+1}}, & \mbox{for the squared-$L_2$-loss},\\
        \varepsilon^{\frac{2\beta - 1}{2\beta+1}}, & \mbox{for the $L_{\infty}$-loss}.
    \end{cases}
\end{equation}

\medskip 
\noindent \textbf{Step 3.}  Combining \eqref{eq-density-nocta} and \eqref{eq-pf-prop6-3}, due to \Cref{general_lower}, we conclude the proof.
\end{proof}

\begin{proof}[Proof of Theorem \ref{L2_risk}] 

To simplify notation, let $\mathbb{E}$ denote the unconditional expectation $\mathbb{E}_{P_{f_\varepsilon}, Q}$.  Using the Sobolev ellipsoid characterisation \eqref{ellipsoid} with trigonometric functions \eqref{trigonometricbasis}, we write 
\[
    f = \sum_{j=1}^\infty \theta_j\varphi_j = \sum_{j=1}^\infty \left\{\int_{[0, 1]} f(x) \varphi_j(x)\, \mathrm{d}x \right\} \varphi_j.
\] 
Due to the orthonormality of the basis functions, we have that
\begin{equation}\label{bvl2}
    \|\hat{f}-f\|_2^2 = \sum_{j=1}^k(\theta_j-\overline{Z}_j)^2 + \sum_{j>k}\theta_j^2  = (I) + (II).
\end{equation}

For term $(II)$, using \eqref{ellipsoid}, we have that 
\begin{equation}\label{tribias}
    \sum_{j>k}\theta_j^2 = \sum_{j>k}j^{2\beta} \frac{\theta_j^2}{j^{2\beta}} \leq \frac{1}{k^{2\beta}}\sum_{j>k}j^{2\beta}\theta_j^2 \leq r^2 k^{-2\beta}.
\end{equation}
For term $(I)$,  note that 
\begin{equation*}
    \mathbb{E}(Z_{i,j}) = \mathbb{E}\Big\{\mathbb{E}_Q(Z_{i,j}|X_i)\Big\} = \mathbb{E}_{P_{f_\varepsilon}}\{\varphi_j(X_i)\} = (1-\varepsilon)\theta_j + \varepsilon \mathbb{E}_G\{\varphi_j(X_i)\}. 
\end{equation*}
Since $\sup_{1 \leq j \leq k}\|\varphi_j(\cdot)\|_{\infty}\leq B_0 = \sqrt{2}$, we have that
\[
    \big|\mathbb{E}\left(\overline{Z}_{j} - \theta_j\right)\big| = \Big|\frac{1}{n}\sum_{i=1}^n \varepsilon(\mathbb{E}_g[\varphi_j(X_i)] - \theta_j) \Big| \leq \varepsilon (B_0+ |\theta_j|).
\]
Since $|Z_{i,j}| = B$, by the construction of privacy mechanism detailed in Appendix \ref{addition4}, for any $i = 1,\dotsc,n $ and $j =1,\dotsc,k$, we have that $\mathrm{Var}[Z_{i,j}]\leq B^2$. The independence of $Z_{1,j},\dotsc,Z_{n,j}$ leads to that
\begin{equation}\label{meanl2}
    \mathbb{E}[(\theta_j-\overline{Z}_j)^2] = \left(\mathbb{E}[\overline{Z}_j] - \theta_j\right)^2 + \mathrm{Var}(Z_{i,j})/n \leq 2\varepsilon^2(B_0^2+\theta_j^2)+B^2/n.
\end{equation}

Combining \eqref{bvl2}, \eqref{tribias} and \eqref{meanl2}, we have that
\begin{align}
    \mathbb{E}(\|\hat{f}-f\|_2^2)  &\leq \sum_{j=1}^k\left\{2\varepsilon^2(B_0^2+\theta_j^2)+\frac{B^2}{n}\right\} + \frac{r^2}{k^{2\beta}} \nonumber \\
    & \lesssim k\varepsilon^2B_0^2 +\varepsilon^2r^2 +\frac{B_0^2k^2}{n\alpha^2}+ \frac{r^2}{k^{2\beta}} \asymp k\varepsilon^2+\frac{k^2}{n\alpha^2} + \frac{1}{k^{2\beta}}, \label{eq-proof-thm7-1}
\end{align}
where the second inequality is due to \eqref{ellipsoid} and the choice of $B$ in the privacy mechanism, and the final derivation is due to $B_0 = \sqrt{2} \asymp r$.

With the choice $k \asymp \varepsilon^{-\frac{2}{2\beta+1}} \wedge (n\alpha^2)^{\frac{1}{2\beta+2}}$, the upper bound in \eqref{eq-proof-thm7-1} reads as
\[
\mathbb{E}[\|\hat{f}-f\|_2^2] \lesssim \varepsilon^{\frac{4\beta}{2\beta+1}} \vee (n\alpha^2)^{-\frac{2\beta}{2\beta+2}},
\]
which concludes the proof.
\end{proof}

\begin{proof}[Proof of Theorem \ref{infinity_risk}]

In this proof, to simplify notation, let $\hat{f}$ denote the estimator $\hat{f}_{\mathrm{Lap}}$ and let $\mathbb{E}_W$ be the expectation with respect to the Laplace random variables and $\mathbb{E}$ for the unconditional expectation $\mathbb{E}_{P_{f_\varepsilon}, Q}$.  

In order to upper bound the $L_{\infty}$-loss assuming the true density function $f \in \mathcal{F}_{\beta}$, in this proof, we upper bound the loss for a larger space of $f$.  To be specific, by the Sobolev embedding theorem \citep[Corollary 9.2 in][]{hardle2012wavelets}, it holds that $\mathcal{F}_{\beta} \subseteq D^{\beta-1/2,\infty}_\infty$, where
\[
D^{\beta-1/2,\infty}_\infty = \left\{f \in B^{\beta-1/2,\infty}_\infty, \, f\geq 0, \, \mathrm{supp}(f) \subseteq [0,1], \, \int_{[0, 1]} f(x)\, \mathrm{d}x = 1 \right\},
\]
with $B^{\beta-1/2,\infty}_\infty$ being a Besov space
\begin{equation}\label{sobem}
    B^{\beta-1/2,\infty}_\infty = \left\{f = \sum_{j\geq -1}\sum_{k \in \mathcal{N}_j}\beta_{jk}\psi_{jk}, \, \sup_{j,k} 2^{j\beta }|\beta_{jk}| < \infty\right\}.
\end{equation}
It then suffices to upper bound $\sup_{f \in D^{\beta-1/2,\infty}_\infty} \mathbb{E}(\|\hat{f} - f\|_{\infty})$.

Since $f \in D^{\beta-1/2,\infty}_\infty$ is a density, using the representation in \eqref{sobem}, the coefficients $\beta_{jk}$'s can only be non-zero if $-A \leq -k \leq 2^j -k \leq A$, and so $|\mathcal{N}_j| \leq 2^j + 2A+1$. For each fixed value of $x$ there are at most $2A+1$ values of $k$ for which $\psi_{jk}(x)$ is non-zero. 

It follows from Lemma \ref{donoho} with $p = \infty$ that
\begin{align*}
    \|\hat{f} - f\|_\infty &\leq \sum_{j = -1}^{J}\sup_{x \in [0,1]} \Bigg|\sum_{k \in \mathcal{N}_j} (\hat{\beta}_{jk}-\beta_{jk})\psi_{jk}(x)\Bigg| + \sum_{j > J}\sup_{x \in [0,1]} \Bigg|\sum_{k \in \mathcal{N}_j} \beta_{jk} \psi_{jk}(x)\Bigg| \\ &\lesssim \sum_{j= -1}^J 2^{j/2} \max_{k \in \mathcal{N}_j}|\hat{\beta}_{jk}-\beta_{jk}|  + \sum_{j = -1}^J 2^{j/2} \max_{k\in \mathcal{N}_j}|\beta_{jk}|. 
\end{align*}
In Lemma \ref{L4}, we show that 
\[
\mathbb{E}\left[\max_{-1\leq j\leq J,k \in \mathcal{N}_j}\Big|\hat{\beta}_{jk} - {\beta}_{jk}\Big|\right] \lesssim \sigma_J\sqrt{\frac{J}{n}} + 2^{J/2}\sqrt{\frac{J}{n}} + \varepsilon 2^{J/2},
\]
which then leads to that
\begin{align*}
    \mathbb{E}[\|\hat{f} - f\|_{\infty}] &\lesssim \sum_{j = -1}^J 2^{j/2} \mathbb{E}\left[\max_{k \in \mathcal{N}_j}|\hat{\beta}_{jk} - \beta_{jk}|\right] + \sum_{j>J}2^{j/2}\max_{k \in \mathcal{N}_j}|\beta_{jk}|\\
    & \lesssim 2^{J/2}\left( \sigma_J\sqrt{\frac{J}{n}} + 2^{J/2}\sqrt{\frac{J}{n}} + \varepsilon 2^{J/2} \right) + \sum_{j>J} 2^{j(1/2-\beta)} \\
    & \lesssim \frac{2^J}{\alpha} \sqrt{\frac{J}{n}} + \varepsilon 2^J + 2^{J(1/2-\beta)},
\end{align*}
where the second line uses \eqref{sobem} and the last line is due to the choice $\sigma_J$ in \eqref{sigmaJ} and $\beta >1/2$. 
With the choice of $2^J \asymp \left(\frac{\log(n\alpha^2)}{n\alpha^2}\right)^{-\frac{1}{2\beta+1}} \wedge \varepsilon^{-\frac{2}{2\beta+1}} $, we have for any $f \in D^{\beta-1/2,\infty}_\infty$
\begin{align*}
    & \mathbb{E}[\|\hat{f} - f\|_{\infty}] \lesssim \frac{2^J}{\alpha} \sqrt{\frac{\log(n\alpha^2)}{n}} \sqrt{\frac{J}{\log(n\alpha^2)}} + \varepsilon 2^J + 2^{J(1/2-\beta)} \\
    \lesssim & \frac{2^J}{\alpha} \sqrt{\frac{\log(n\alpha^2)}{n}} + \varepsilon 2^J + 2^{J(1/2-\beta)} \lesssim \left(\frac{n\alpha^2}{\log(n\alpha^2)}\right)^{-\frac{2\beta-1}{4\beta+2}} \vee \varepsilon^{\frac{2\beta-1}{2\beta+1}},
\end{align*}
since $2^J \lesssim \left({n\alpha^2}\right)^{\frac{1}{2\beta+1}}$ and $J \lesssim \log(n\alpha^2)$. 

We then conclude the proof.
\end{proof}

\begin{lemma}[Lemma 1 in \cite{donoho1996density}]\label{donoho} Let $\theta(x) = \theta_\psi(x) = \sum_{k \in \mathbb{Z}}|\psi(x-k)|$ and $f(x) = \sum_{k \in \mathbb{Z}}\lambda_k 2^{j/2}\psi(2^jx-k)$. For $1 \leq p \leq \infty$, then 
\[
2^{j(1/2-1/p)}\|\lambda\|_{l_p} \lesssim\|f\|_p \lesssim 2^{j(1/2-1/p)}\|\lambda\|_{l_p}
\]
\end{lemma}

\begin{lemma}\label{sub-exp}
Let $\tilde{\beta}_{jk} = \frac{1}{n}\sum_{i=1}^n \psi_{jk}(X_i)$, then
\[
\mathbb{E}_W\left[\max_{-1\leq j\leq J,k \in \mathcal{N}_j}\Big|\hat{\beta}_{jk} - \tilde{\beta}_{jk}\Big|\right] \lesssim \sigma_J\sqrt{\frac{J}{n}}
\]
\end{lemma}
\begin{proof}
Note that $\hat{\beta}_{jk} - \tilde{\beta}_{jk} = \frac{1}{n}\sigma_J \sum_{i=1}^nW_{ijk}$ and $\sigma_JW_{ijk}$ is a sub-exponential random variable with parameters ($2 \sigma_J, 2\sigma_J$) according to the definition of sub-Exponential random variables. Since linear combination preserves sub-exponential property \citep[e.g.\ Theorem 2.8.2 in][]{vershynin2018high}, we have for each fixed $-1\leq j \leq J, k \in \mathcal{N}_j$, $\hat{\beta}_{jk} - \tilde{\beta}_{jk}$ is a sub-exponential random variable with parameters $\left(\frac{
2\sigma_J}{\sqrt{n}},\frac{2\sigma_J}{n}\right)$. Now, applying standard maximal inequality for sub-exponential random variables \citep[e.g.\ Corollary 2.6 in][]{boucheron2013concentration} yields 
\[
\mathbb{E}_W\left[\max_{-1\leq j\leq J,k \in \mathcal{N}_j}\Big|\hat{\beta}_{jk} - \tilde{\beta}_{jk}\Big|\right] \lesssim \sigma_J \sqrt{\frac{\log(S)}{n}} + \sigma_J \frac{\log(S)}{n}
\]
where $S = \sum_{j \leq J}|\mathcal{N}_j|$ is the total number of terms considered. Since each $|\mathcal{N}_j| \leq 2^j + 2A + 1$ we have $S \lesssim 2^J$, which leads to the claimed result. 
\end{proof}

\begin{lemma}\label{L4}

\[
\mathbb{E}\left[\max_{-1\leq j\leq J,k \in \mathcal{N}_j}\Big|\hat{\beta}_{jk} - {\beta}_{jk}\Big|\right] \lesssim \sigma_J\sqrt{\frac{J}{n}} + 2^{J/2}\sqrt{\frac{J}{n}} + \varepsilon 2^{J/2}
\]
\end{lemma}
\begin{proof}
A direct decomposition leads to 
\begin{align*}
& \max_{-1\leq j\leq J,k \in \mathcal{N}_j}\Big|\hat{\beta}_{jk} - {\beta}_{jk}\Big| \\
\leq & \max_{-1\leq j\leq J,k \in \mathcal{N}_j}\Big|\hat{\beta}_{jk} - \tilde{\beta}_{jk}\Big| + \max_{-1\leq j\leq J,k \in \mathcal{N}_j}\Big|\tilde{\beta}_{jk} - \mathbb{E}[\tilde{\beta}_{jk}]\Big|+\max_{-1\leq j\leq J,k \in \mathcal{N}_j}\Big|\mathbb{E}[\tilde{\beta}_{jk}] - {\beta}_{jk}\Big|.
\end{align*}
The first term is dealt with in Lemma \ref{sub-exp}. For the second term, we notice that for each $\tilde{\beta}_{jk}$, it has sub-Gaussian parameter bounded by $\frac{2^{j/2}\|\psi\|_{\infty}}{\sqrt{n}} \vee \frac{\|\phi\|_\infty}{\sqrt{n}}$. We assume without loss of generality that $\|\phi\|_\infty \leq \|\psi\|_\infty$ and use maximal inequality for sub-Gaussian random variables \citep[e.g.\ Theorem 2.5 in][]{boucheron2013concentration} to obtain 
\begin{equation}\label{second-term}
    \mathbb{E}_{f_\varepsilon}\left[\max_{-1\leq j\leq J,k \in \mathcal{N}_j}\Big|\tilde{\beta}_{jk}- \mathbb{E}[\tilde{\beta}_{jk}]\Big|\right] \lesssim 2^{J/2}\|\psi\|_{\infty}\sqrt{\frac{J}{n}}.
\end{equation}
For the last term, we note that by the definition of $D^{\beta-1/2,\infty}_\infty$, we have $\max_{-1\leq j\leq J,k \in \mathcal{N}_j} |\beta_{jk}| \lesssim 1$. Therefore, we have 
\begin{align}
    & \max_{-1\leq j\leq J,k \in \mathcal{N}_j}\Big|\mathbb{E}_{f_\varepsilon}[\tilde{\beta}_{jk}] - {\beta}_{jk}\Big| = \max_{-1\leq j\leq J,k \in \mathcal{N}_j} \varepsilon  \Big|\mathbb{E}_g[\tilde{\beta}_{jk}] -  {\beta}_{jk}\Big| = \max_{-1\leq j\leq J,k \in \mathcal{N}_j} \varepsilon \Big|\mathbb{E}_g[\psi_{jk}] - {\beta}_{jk}\Big| \nonumber \\
    \leq & \max_{-1\leq j\leq J,k \in \mathcal{N}_j} \varepsilon \Big|\mathbb{E}_g[\psi_{jk}]\Big| + \varepsilon\max_{-1\leq j\leq J,k \in \mathcal{N}_j} |\beta_{jk}| \leq \varepsilon 2^{J/2}\|\psi\|_{\infty} \label{last}
\end{align} 
Combining (\ref{second-term}), \eqref{last} and Lemma \ref{sub-exp}, we obtain the claimed result. 

\end{proof}

\subsection[]{Additional details in \Cref{sec3}}\label{addition4}

We provide the detailed privacy mechanism used to generate \eqref{fhat}.  This is from Section 4.2.3 in \cite{duchi2018minimax}, see where for more details.  We present it here for completeness.  Given a vector $v \in \mathbb{R}^k$ with $\|v\|_{\infty} \leq B_0 = \sqrt{2}$, generate the private version $Z$ of~$v$ in two steps.  Recall that $\sqrt{2}$ is an upper bound of the supremum norm of the trignometric basis functions.  Generally speaking, for uniformly bounded basis functions, $B_0$ corresponds to the supremum norm of the chosen basis functions.
\begin{itemize}
    \item [Step 1.]  Generate $\widetilde{V}$ according to
    \[
        \mathbb{P}(\tilde{V} = B_0) = 1 - \mathbb{P}(\tilde{V} = -B_0) = \frac{1}{2}+\frac{v_j}{2B_0}.
    \]
    \item [Step 2.]  Let $T$ be a Bernoulli($e^\alpha/(e^\alpha+1)$) random variable independent of $\widetilde{V}$ and the data.  Generate $Z$ according to 
    \[
        Z \sim \begin{cases}
            \mathrm{Uniform}\Big(z \in \{-B,B\}^k \Big|\langle z,\tilde{V} \rangle \geq 0\Big), & T = 1, \\
            \mathrm{Uniform}\Big(z \in \{-B,B\}^k \Big|\langle z,\tilde{V} \rangle \leq 0\Big), & T = 0,
    \end{cases}
    \]
    where 
    \[
        B = B_0C_d \frac{e^\alpha+1}{e^\alpha-1} \quad \text{and} \quad C_d^{-1} =\begin{cases}  \frac{1}{2^d-1} \binom{d-1}{(d-1)/2}, & d \equiv 1 \Mod 2, \\
        \frac{2}{2^d+ \binom{d}{d/2}} \binom{d-1}{d/2}, & d \equiv 0 \Mod 2. 
        \end{cases}
    \]
\end{itemize}

\section{Robust one-dimensional median estimation under local differential privacy}
\label{secE}

Recall the general setup in \Cref{sec-general-setup}, in this section we consider distributions supported on the real line, i.e.~$\mathcal{X} = \mathbb{R}$, with finite first central moment and median lying in $[-r, r]$.  We are interested in estimating the median, i.e.~$\theta(P) = \mathrm{med}(P)$.  To be specific, we let
\begin{equation}\label{eq-median-classr}
       \mathcal{P}_r = \{\mbox{distribution } P \mbox{ such that  } |\mathrm{med}(P)| \leq r, \, \mathbb{E}_P\{|X|\} < \infty\},
\end{equation}
where $r > 0$ is potentially a function of the sample size $n$.

This problem was studied in Section~3.2.2 in \cite{duchi2018minimax}, after observing that the robust mean estimation studied in Section 3.2.1 in \cite{duchi2018minimax} (and a robust counterpart studied in \Cref{sec2}) becomes harder when the moment parameter $k$ decreases.  In \cite{duchi2018minimax}, when the moment parameter $k = 1$, i.e.~a potentially rather heavy-tailed case, the focus is shifted to the median estimation, instead of estimating mean -- a non-robust quantity.  The privacy mechanism studied thereof is a sequentially interactive stochastic gradient decent method.  Since median estimation is impossible for general distribution, the minimax optimality is studied based on an excess risk.

Medians are naturally robust quantities, especially robust against heavy-tailedness.  In this section, we carry on the narrative of this paper, estimating medians subject to both the privacy constraints and Huber's contamination.  To be specific, we assume the data are generated from Huber's contamination model \eqref{hubermodel}.  Let $\{Y_i\}_{i = 1}^n$ be i.i.d.~random variables with distribution $P_{r, \varepsilon} \in \mathcal{P}_{\varepsilon}(\mathcal{P}_r)$ and suppose that we are interested in estimating the median of the inlier distribution $\theta$.  The $\alpha$-LDP minimax risk takes its specific form in the robust median estimation problem as follows
    \begin{align}
        \mathcal{R}_{n, \alpha}(\varepsilon) & = \inf_{Q \in \mathcal{Q}_{\alpha}} \inf_{\widehat{\theta}} \sup_{\substack{P_{r, \varepsilon} \in \mathcal{P}_{\varepsilon} (\mathcal{P}_r)}} \mathbb{E}_{P_{r, \varepsilon}, Q} \big\{R(\widehat{\theta}) - R(\theta(P))\big\} \nonumber \\
        & = \inf_{Q \in \mathcal{Q}_{\alpha}} \inf_{\widehat{\theta}} \sup_{P_{r, \varepsilon} \in \mathcal{P}_{\varepsilon} (\mathcal{P}_r)} \mathbb{E}_{P_{r, \varepsilon}, Q} \big\{\mathbb{E}_{X \sim P}\{|X - \widehat{\theta}(Z_1, \ldots, Z_n)|\} - \mathbb{E}_{X \sim P}\{|X - \theta(P)|\}\big\}, \label{eq-excess-risk-minimax}
    \end{align}
where 
    \begin{equation}\label{eq-excess-risk-def}
        R(\cdot) = \mathbb{E}_{X \sim P}\{|X-\cdot|\}
    \end{equation} 
    is the excess risk defined on the inlier distribution $P$, the infimum over $\widehat{\theta}$ is taken over all measurable functions of the privatised data generated from any Huber contamination model $P_{r, \varepsilon}$ and some privacy mechanism $Q \in \mathcal{Q}_{\alpha}$.  Note that this is slightly different from the $\alpha$-LDP minimax risk defined \eqref{intro_minimax}, since we are interested in the excess risk here.

\subsection{Lower bound}\label{sec-lb-median}

In \Cref{prop-medianlowerbound} we show that, if either the privacy cost is high that $n\alpha^2 \lesssim r$ or the contamination proportion is high that $r \varepsilon \gtrsim 1$, then $\mathcal{R}_{n, \alpha}(\varepsilon) \gtrsim 1$.

\begin{prop}\label{prop-medianlowerbound}
Let $\{Y_i\}_{i = 1}^n$ be i.i.d.~random variables from $P_{r, \varepsilon} \in \mathcal{P}_{\varepsilon}(\mathcal{P}_r)$, with $\mathcal{P}_r$ defined in~\eqref{eq-median-classr}.  For $\alpha \in (0, 1]$ and $r > 0$, it holds that the $\alpha$-LDP minimax excess risk defined in \eqref{eq-excess-risk-minimax} satisfies 
\[
\mathcal{R}_{n, \alpha}(\varepsilon) \gtrsim \frac{r}{\sqrt{n \alpha^2}} \vee r\varepsilon. 
\]
\end{prop}

We reiterate that the minimax risk considered here is regarding a form of excess risk, rather than the estimation risk.  This handicaps a direct application of \Cref{general_lower} on the part of the lower bound when there is no privacy constraints.  We provide a detailed proof in \Cref{sec-median-proofs}.  Nevertheless, the result of \Cref{prop-medianlowerbound} shares the same spirit of \Cref{general_lower}, that the minimax risk can be disentangled into the case of privacy and contamination.  The term $r(n\alpha^2)^{-1/2}$ is indeed the minimax rate shown in \cite{duchi2018minimax} for the case when there is only privacy constraints.

\subsection{Upper bound}   

In this section, we consider the stochastic gradient descent (SGD) procedure as used in \cite{duchi2018minimax}.  It is shown in \cite{duchi2018minimax} that this sequentially interactive privacy mechanism is optimal when estimating medians.  In view of Q2 we raised in \Cref{sec-contributions}, we will show that this locally private procedure is automatically robust in the univariate median estimation problem.

Let $\theta_1 \in [-r, r]$ be arbitrary and $\{W_i\}_{i = 1}^n$ be a sequence of i.i.d.~random variables with 
    \[
        \mathbb{P}\{W_1 = 1\} = 1 - \mathbb{P}\{W_1 = -1\} = \frac{e^{\alpha}}{e^{\alpha} + 1}.
    \]
    The SGD procedure iterates according to 
    \begin{equation}\label{eq-sgd-ldp}
        \theta_{i+1} = [\theta_i - \eta_i Z_i]_r, 
    \end{equation}
    where
    \[
        Z_i = \frac{e^{\alpha}+1}{e^{\alpha}-1} W_i s(\theta_i - Y_i) \quad \mbox{and} \quad s(\theta_i - Y_i) \begin{cases}   
            = \mathrm{sign}(\theta_i - Y_i), & \theta_i \neq Y_i, \\
            \sim \mathrm{Unif}\{-1, 1\}, & \theta_i = Y_i,
        \end{cases}
    \]    
    with the truncation operator $[\cdot]_r = \max\{-r, \, \min\{\cdot, \, r\}\}$ and non-increasing step sizes sequence $\eta_i > 0$, $i = 1, \ldots, n$.  Define the estimator $\widehat{\theta}$ as
    \begin{equation}\label{eq-median-thetahat}
         \widehat{\theta} = \frac{1}{\sum_{j = 1}^n \eta_j} \sum_{i = 1}^n \eta_i \theta_i.
    \end{equation}

\begin{remark} 
We have that $\{Z_i\}_{i = 1}^n$ satisfy $\alpha$-LDP, by noticing that
\[
    \mathbb{P}\left\{Z_i = \frac{e^{\alpha} + 1}{e^{\alpha} - 1}\bigg| \theta_i(Z_{1:(i-1)}), Y_i\right\} = \frac{e^{\alpha}}{e^{\alpha} + 1} \mathbbm{1}\{\theta_i > Y_i\} + \frac{1}{e^{\alpha} + 1} \mathbbm{1}\{\theta_i < Y_i\} + \frac{1}{2} \mathbbm{1}\{\theta_i = Y_i\}
\]
and
\[
    \mathbb{P}\left\{Z_i = - \frac{e^{\alpha} + 1}{e^{\alpha} - 1}\bigg| \theta_i(Z_{1:(i-1)}), Y_i\right\} = \frac{1}{e^{\alpha} + 1} \mathbbm{1}\{\theta_i > Y_i\} + \frac{e^{\alpha}}{e^{\alpha} + 1} \mathbbm{1}\{\theta_i < Y_i\} + \frac{1}{2} \mathbbm{1}\{\theta_i = Y_i\}. 
\]
\end{remark}

\begin{prop}\label{prop-medianupperbound}
Given i.i.d.~random variables $\{Y_i\}_{i=1}^n$ with distribution $P_{r, \varepsilon} = (1-\varepsilon)P+\varepsilon G$ where $P \in \mathcal{P}_r$ defined in \eqref{eq-median-classr}, $G$ is an arbitrary distribution supported on $\mathbb{R}$ and $\varepsilon \in (0,1)$.  For $\alpha \in (0, 1)$, the estimator $\widehat{\theta}$ defined in \eqref{eq-median-thetahat} with constant step size $\eta_i = \alpha r /\sqrt{n}$, $i= 1, \ldots, n$, satisfies that 
\[
\mathbb{E}_{\{Y_i, W_i\}_{i = 1}^n} \big\{R(\widehat{\theta}) - R(\theta(P))\big\} \lesssim \frac{r}{\sqrt{n\alpha^2}} \vee \varepsilon r,
\]
where $R(\cdot)$ is the excess risk defined in \eqref{eq-excess-risk-def}.
\end{prop}

The upper bound we obtained in \Cref{prop-medianupperbound}, combining with the lower bound result in \Cref{prop-medianlowerbound}, shows that the private SGD procedure retains minimax optimality in the presence of contamination, in a univariate median estimation problem.

\subsection[]{Proofs of the results in \Cref{secE}}\label{sec-median-proofs}

\begin{proof}[Proof of \Cref{prop-medianlowerbound}]
It follows from Corollary 2 in \cite{duchi2018minimax} that $\mathcal{R}_{n, \alpha}(\varepsilon) \gtrsim r/\sqrt{n\alpha^2}$ when $\varepsilon = 0$. It remains to show the other part in the lower bound. 

Consider two distributions
\[
    R_0: \, \mathbb{P}(X = r) = \frac{1 + \delta}{2} = 1 - \mathbb{P}(X = -r)
\]
and
\[
    R_1: \, \mathbb{P}(X = r) = \frac{1 - \delta}{2} = 1 - \mathbb{P}(X = -r).
\]
We have that 
\[
    \theta(R_0) = r, \quad \theta(R_1) = -r \quad \mbox{and} \quad \mathrm{TV}(R_0, R_1) = \delta.
\]
For such $R_0,R_1$ the risk $\mathbb{E}_{X \sim R}(|X-\theta|)$ is always minimised for $\theta \in [-r,r]$, and for such $\theta$ we have
\[
    \mathbb{E}_{X \sim R_0}(|X-\theta|) = r-\delta \theta \quad \text{and} \quad  \mathbb{E}_{X \sim R_1}(|X-\theta|) = r+\delta \theta.
\]
Note also that we can write
\[
    R_0 = (1-\varepsilon) R_1 + \varepsilon P_{r},
\]
where $P_{r}$ is a point mass at $r$ and $\varepsilon = 2\delta/(1+\delta)$. Here $\delta=\varepsilon/(2-\varepsilon)$ and it follows that
\begin{align*}
    &\mathcal{R}_{n, \alpha}(\varepsilon) \geq \inf_{\hat{\theta}} \max \bigl[ \mathbb{E}_{Y_1,\ldots,Y_n \sim R_0} \mathbb{E}_{X \sim R_0}\{ |\hat{\theta}(Y_1,\ldots,Y_n) - X| - |r-X|), \\
    & \hspace{120pt} \mathbb{E}_{Y_1,\ldots,Y_n \sim R_0} \mathbb{E}_{X \sim R_1}\{|\hat{\theta}(Y_1,\ldots,Y_n) - X| - |r+X|\} \bigr] \\
    & = \inf_{\hat{\theta}} \max \bigl\{ \mathbb{E}_{Y_1,\ldots,Y_n \sim R_0} \mathbb{E}_{X \sim R_0}  |\hat{\theta}(Y_1,\ldots,Y_n) - X|, \\
    &\hspace{70pt} \mathbb{E}_{Y_1,\ldots,Y_n \sim R_0} \mathbb{E}_{X \sim R_1}  |\hat{\theta}(Y_1,\ldots,Y_n) - X| \bigr\} - r(1-\delta) \\
    &\geq \inf_{\hat{\theta}} \max \bigl\{ \mathbb{E}_{Y_1,\ldots,Y_n \sim R_0} \mathbb{E}_{X \sim R_0}  |[\hat{\theta}(Y_1,\ldots,Y_n)]_{-r}^r - X|, \\
    &\hspace{70pt} \mathbb{E}_{Y_1,\ldots,Y_n \sim R_0} \mathbb{E}_{X \sim R_1}  |[\hat{\theta}(Y_1,\ldots,Y_n)]_{-r}^r - X| \bigr\} - r(1-\delta) \\
    & = \inf_{\hat{\theta}} \max \bigl\{ r - \delta \mathbb{E}_{Y_1,\ldots,Y_n \sim R_0} ([\hat{\theta}(Y_1,\ldots,Y_n)]_{-r}^r) , r+\delta \mathbb{E}_{Y_1,\ldots,Y_n \sim R_0} ([\hat{\theta}(Y_1,\ldots,Y_n)]_{-r}^r) \bigr\} - r(1-\delta) \\
    & = r\delta + \delta \max \bigl\{ -\mathbb{E}_{Y_1,\ldots,Y_n \sim R_0} ([\hat{\theta}(Y_1,\ldots,Y_n)]_{-r}^r), \mathbb{E}_{Y_1,\ldots,Y_n \sim R_0} ([\hat{\theta}(Y_1,\ldots,Y_n)]_{-r}^r) \bigr\} \geq r \delta = \frac{r \varepsilon}{2-\varepsilon}.
\end{align*}
We therefore conclude the proof.
\end{proof}

\begin{proof}[Proof of \Cref{prop-medianupperbound}]
The analysis of SGD procedure follows similar lines as in \cite{nemirovski2009robust}.
For notational simplicity, we let $\mathbb{E}_Y$ denote the expectation with respect to $(P_{r, \varepsilon}, Q)$, $\mathbb{E}_X$ denote the expectation with respect to $P \in \mathcal{P}_r$ and write $\theta$ for $\theta(P)$.  For $j \in \{1, \ldots, n\}$, let 
\[
    A_j = 2^{-1}(\theta_j - \theta)^2 \quad a_j = \mathbb{E}_Y\{A_j\} = 2^{-1}\mathbb{E}\{(\theta_j - \theta)^2\}.
\]
Let $\Pi_r(\cdot)$ be the projection to $[-r, r]$.  Since $\theta \in [-r, r]$, we have that $\Pi_r(\theta) = \theta$.  It holds that
\begin{align}
    A_{j + 1} & = 2^{-1}\{\Pi_r(\theta_j - \eta_jZ_j) - \theta\}^2 = 2^{-1}\{\Pi_r(\theta_j - \eta_jZ_j) - \Pi_r(\theta)\}^2 \nonumber \\
    & \leq 2^{-1}(\theta_j - \eta_jZ_j - \theta)^2 = A_j + 2^{-1}\eta_j^2 Z_j^2 - \eta_j (\theta_j - \theta)Z_j. \label{eq-Aj1-Aj}
\end{align}
Since $\theta_j$ is independent of $Y_j$, we have that
\begin{align*}
    \mathbb{E}_Y\{(\theta_j - \theta)Z_j\} & = \mathbb{E}_{Y} \{(\theta_j - \theta)s(\theta_j - Y_j)\} = \mathbb{E}_Y \{(\theta_j - \theta) \mathbb{E}[ s(\theta_j - Y_j)| \theta_j]\} \\
    & = \mathbb{E}_Y[(\theta_j - \theta) \{\mathbb{P}_{Y_j|\theta_j}(Y_j < \theta_j) - \mathbb{P}_{Y_j|\theta_j} (Y_j > \theta_j)\}].
\end{align*}
In addition we have that 
\[
    \mathbb{E} \{Z_j^2\} = \left(\frac{e^{\alpha} + 1}{e^{\alpha} - 1}\right)^2, \quad j \in \{1, \ldots, n\}.
\]
Taking expectation of both sides of \eqref{eq-Aj1-Aj}, we obtain that 
\[
    a_{j+1} \leq a_j - \eta_j \mathbb{E}_Y[(\theta_j - \theta) \{\mathbb{P}_{Y_j|\theta_j}(Y_j < \theta_j) - \mathbb{P}_{Y_j|\theta_j} (Y_j > \theta_j)\}] + 2^{-1}\eta_j^2 \left(\frac{e^{\alpha} + 1}{e^{\alpha} - 1}\right)^2,
\]
which is equivalent to
\begin{equation}\label{eq-1}
    \eta_j \mathbb{E}_Y[(\theta_j - \theta) \{\mathbb{P}_{Y_j|\theta_j}(Y_j < \theta_j) - \mathbb{P}_{Y_j|\theta_j} (Y_j > \theta_j)\}] \leq a_j - a_{j+1} + 2^{-1}\eta_j^2 \left(\frac{e^{\alpha} + 1}{e^{\alpha} - 1}\right)^2.
\end{equation}

It follows from the convexity of the absolute value function, we have that 
\[
    \mathbb{E}_X\{|\theta_1 - X|\} \geq \mathbb{E}_X\{|\theta_2 - X|\} + (\theta_1 - \theta_2) \mathbb{E}_X\{\mathrm{sign}(\theta_2 - X)\}.
\]
It then holds that
\[
    \mathbb{E}_Y[(\theta_j - \theta)\mathbb{E}_X\{\mathrm{sign}(\theta_j - X)\}] \geq \mathbb{E}_Y \{\mathbb{E}_X\{|\theta_j - X|\} -  \mathbb{E}_X\{|\theta - X|\}\},
\]
which is equivalent to
\begin{equation}\label{eq-2}
    \eta_j \mathbb{E}_Y[(\theta_j - \theta) \{\mathbb{P}_{X|\theta_j}(X < \theta_j) - \mathbb{P}_{X|\theta_j} (X > \theta_j)\}] \geq \eta_j \mathbb{E}_Y \{\mathbb{E}_X\{|\theta_j - X|\} -  \mathbb{E}_X\{|\theta - X|\}\}.
\end{equation}

Note that
\begin{align}
    & \mathbb{E}_Y[(\theta_j - \theta) \{\mathbb{P}_{Y_j|\theta_j}(Y_j < \theta_j) - \mathbb{P}_{Y_j|\theta_j} (Y_j > \theta_j)\}] \nonumber \\
    = & (1-\varepsilon) \mathbb{E}_Y[(\theta_j - \theta) \{\mathbb{P}_{X|\theta_j}(X < \theta_j) - \mathbb{P}_{X|\theta_j}(X > \theta_j)\}] \nonumber \\
    & \hspace{1cm} + \varepsilon \mathbb{E}_Y[(\theta_j - \theta) \{\mathbb{P}_{R \sim G}(R < \theta_j) - \mathbb{P}_{R \sim G}(R > \theta_j)\}], \label{eq-3}
\end{align}
where $G$ is any arbitrary distribution on $\mathbb{R}$. Combining \eqref{eq-1}, \eqref{eq-2} and \eqref{eq-3}, we have that
\begin{align*}
    & \eta_j \mathbb{E}_Y \{\mathbb{E}_X\{|\theta_j - X|\} -  \mathbb{E}_X\{|\theta - X|\}\} \\
    \leq & (1 - \varepsilon)^{-1} \left\{a_j - a_{j+1} + 2^{-1}\eta_j^2 \left(\frac{e^{\alpha} + 1}{e^{\alpha} - 1}\right)^2\right\} \\
    & \hspace{1cm} + \frac{\varepsilon \eta_j}{1 - \varepsilon} \left|\mathbb{E}_Y[(\theta_j - \theta) \{\mathbb{P}_{R \sim G}(R < \theta_j) - \mathbb{P}_{R \sim G}(R > \theta_j)\}]\right| \\
    \leq & (1 - \varepsilon)^{-1} \left\{a_j - a_{j+1} + 2^{-1}\eta_j^2 \left(\frac{e^{\alpha} + 1}{e^{\alpha} - 1}\right)^2\right\} + \frac{2\varepsilon r \eta_j}{1 - \varepsilon}.
\end{align*}
It then follows that
\begin{align*}
    & \frac{1}{\sum_{j = 1}^n \eta_j}\sum_{i = 1}^n \eta_i \left[\mathbb{E}_{X, Y}\{|X - \theta_i|\} - \mathbb{E}_X\{|X - \theta|\}\right] \\
    \leq & \frac{1}{(1-\varepsilon) \sum_{j = 1}^n \eta_j}\sum_{i = 1}^n  \left\{a_i - a_{i+1} + 2^{-1}\eta_i^2 \left(\frac{e^{\alpha} + 1}{e^{\alpha} - 1}\right)^2\right\} + \frac{2\varepsilon r }{1-\varepsilon} \\
    \leq & \frac{1}{(1-\varepsilon) \sum_{j = 1}^n \eta_j} \left\{a_1 + 2^{-1}\left(\frac{e^{\alpha} + 1}{e^{\alpha} - 1}\right)^2\sum_{i = 1}^n \eta_i^2\right\} + \frac{2\varepsilon r }{1-\varepsilon} \\
    \leq & \frac{2}{(1-\varepsilon) \sum_{j = 1}^n \eta_j} \left\{r^2 + \left(\frac{e^{\alpha} + 1}{e^{\alpha} - 1}\right)^2\sum_{i = 1}^n \eta_i^2\right\} + \frac{2\varepsilon r }{1-\varepsilon}.
\end{align*}
Taking the step sizes to be $\eta_i = \alpha r/\sqrt{n}$, $i = 1, \ldots, n$, we have that 
\[
    \frac{1}{\sum_{j = 1}^n \eta_j}\sum_{i = 1}^n \eta_i \left[\mathbb{E}_{X, Y}\{|X - \theta_i|\} - \mathbb{E}_X\{|X - \theta|\}\right] \leq \frac{2r}{1 - \varepsilon} \left(\frac{1+(e+1)^2}{\sqrt{n\alpha^2}} + \varepsilon\right),
\]
for $\alpha \leq 1$. Finally notice that
\begin{align*}
    & \mathbb{E}_{X, Y}\{|X - \widehat{\theta}|\} - \mathbb{E}_X\{|X - \theta|\} = \mathbb{E}_{X, Y}\left\{\left|\frac{1}{\sum_{j = 1}^n \eta_j} \sum_{i = 1}^n \eta_i (X - \theta_i)\right|\right\} - \mathbb{E}_X\{|X - \theta|\} \\
    \leq & \frac{1}{\sum_{j = 1}^n \eta_j}\sum_{i = 1}^n \eta_i \left[\mathbb{E}_{X, Y}\{|X - \theta_i|\} - \mathbb{E}_X\{|X - \theta|\}\right] \leq \frac{2r}{1 - \varepsilon} \left(\frac{17}{\sqrt{n\alpha^2}} + \varepsilon\right).
\end{align*}
We therefore conclude the proof.
\end{proof}

\end{document}